\documentclass[12pt]{amsart}

\usepackage{latexsym, amssymb, amsmath,ulem,soul,esint}

\usepackage{amsfonts, graphicx}
\usepackage{graphicx,color}
\newcommand{\be}{\begin{equation}}
\newcommand{\ee}{\end{equation}}
\newcommand{\beq}{\begin{eqnarray}}
\newcommand{\eeq}{\end{eqnarray}}
\usepackage{hyperref}
\usepackage{verbatim}
\makeatletter
\renewcommand*{\eqref}[1]{%
  \hyperref[{#1}]{\textup{\tagform@{\ref*{#1}}}}%
}

\newtheorem{thm}{Theorem}[section]

\newtheorem{lma}{Lemma}[section]
\newtheorem{prop}{Proposition}[section]
\newtheorem{cor}{Corollary}[section]

\theoremstyle{remark}
\newtheorem{rem}{Remark}[section]
\numberwithin{equation}{section}

\def\be{\begin{equation}}
\def\ee{\end{equation}}
\def\bee{\begin{equation*}}
\def\eee{\end{equation*}}

\def\lf{\left}
\def\ri{\right}
\newcommand{\de}{\partial}

\newcommand{\Ric}{\mathrm{Ric}}

\def\Ric{\text{\rm Ric}}

\def\la{\langle}
\def\ra{\rangle}
\def\p{\partial}

\def\e{\varepsilon}
\def\a{{\alpha}}
\def\b{{\beta}}

\def\R{\mathbb{R}}

\def\mH{\mathcal{H}}
\def\mF{\mathcal{F}}
\def\Na{\nabla}
\def\S{\Sigma}
\def\div{\mathrm{div}}
\def\loc{\mathrm{loc}}

\begin{document}

\title[]
{Monotonicity of the $p$-Green functions}

 \author{Pak-Yeung Chan}
\address[Pak-Yeung Chan]{Department of Mathematics, University of California, San Diego, La Jolla, CA 92093}
\email{pachan@ucsd.edu}

\author{Jianchun Chu}
\address[Jianchun Chu]{School of Mathematical Sciences, Peking University, Yiheyuan Road 5, Beijing, P.R.China, 100871}
\email{jianchunchu@math.pku.edu.cn}

\author{Man-Chun Lee}
\address[Man-Chun Lee]{Department of Mathematics, The Chinese University of Hong Kong, Shatin, N.T., Hong Kong }
\email{mclee@math.cuhk.edu.hk}

\author{Tin-Yau Tsang}
\address[Tin-Yau Tsang]{Department of Mathematics, University of California, Irvine, CA 92697 }
\email{tytsang@uci.edu}

\renewcommand{\subjclassname}{
\textup{2020} Mathematics Subject Classification}

\date{\today}

\begin{abstract} On a complete  $p$-nonparabolic $3$-dimensional manifold with non-negative scalar curvature and vanishing second homology, we establish a sharp monotonicity formula for the proper $p$-Green function along its level sets for $1<p<3$. This can be viewed as a generalization of the recent result by Munteanu-Wang \cite{MunteanuWang2021} in the case of $p=2$. No smoothness assumption is made on the $p$-Green function when $1<p\leq 2$. Several rigidity results are also proven.
\end{abstract}


\maketitle

\section{Introduction}
The study of monotonic quantities plays a crucial role in Riemannian geometry. One fundamental example is the classical Bishop-Gromov volume comparison theorem which says that on a complete manifold $M^n$ with $\Ric\ge (n-1)K$, the function
\begin{equation}
\frac{\text{Vol}\left(B_r(p)\right)}{\text{V}_K(r)}
\end{equation}
is monotone non-increasing in $r$, where $K\in \R$ and $\text{V}_K(r)$ denotes the volume of the ball of radius $r$ in simply connected space form of constant curvature $K$. The monotone property does not only provide bounds for the quantity, but also rules out the complexity of the function and guarantees the existence of the limit of the monotonic quantity as its parameter approaches a certain number, which provides a lot of geometric information, e.g. asymptotic volume ratio.
On a complete non-parabolic manifold $M^n$ with non-negative Ricci curvature, Colding \cite{Colding2012} proved several monotonicity formulas for the function
\begin{equation}\label{old mono}
r^{1-n}\int_{\{b=r\}} |\nabla b|^{\beta}\, dA,
\end{equation}
where $\beta=3$, $b=G^{1/(2-n)}$ and $G$ is the Green function of the Laplacian $\Delta$ (see also \cite{ColdingMinicozzi2014} for the generalization to the cases of $\beta\ge \frac{2n-3}{n-1})$. The result was then used by Colding-Minicozzi \cite{ColdingMinicozzi2014b} to study the uniqueness of tangent cones of Einstein manifolds. Monotonic quantities are also important in the field of PDEs, for instance, Garofaro-Lin \cite{GarofaloLin1986, GarofaloLin1987} proved the unique continuation property of elliptic operator by showing the monotonicity of an Almgren frequency type function, see also \cite{AFM20, AFM19, AM17, AM20, Colding2012, ColdingMinicozzi2013, ColdingMinicozzi2014, MunteanuWang2021} and references therein for detailed accounts on the monotonicity formulas and their applications.

\medskip

Recently, Stern \cite{Stern2019} has developed a new technique to connect the geometry of $M$ to the level sets of its harmonic functions or harmonic forms. This provides a new pathway to study complete $3$-dimensional manifolds with non-negative scalar curvature $R_{M}\ge 0$. The method is far-reaching and has been proven to be successful in view of the recent progresses in \cite{AMO2021, BKKS2019, ChodoshLi2021, MunteanuWang2021, MunteanuWang2022}. In particular, Bray-Kazaras-Khuri-Stern \cite{BKKS2019} introduced a new formula relating  certain integral of the harmonic functions to the ADM mass of an asymptotically flat manifold and give an alternative proof to the positive mass theorem in dimension $3$ (see also \cite{AMO2021} for a proof using monotonicity of harmonic function). Motivated by the results in \cite{Colding2012, ColdingMinicozzi2014}, Munteanu-Wang \cite{MunteanuWang2021, MunteanuWang2022} investigated the quantity
\begin{equation}\label{mono quantity}
t^{-1}\int_{\{z:\, G(x_0,z)=t\}}|\nabla G(x_0,\cdot)|^2\, dA-4\pi t,
\end{equation}
where $G(x_0,\cdot)$ is the Green function based at $x_0\in M$. They showed that the above quantity is non-increasing with respect to $t$, established different comparison theorems under scalar curvature and rough Ricci curvature bounds and gave some geometric applications. The quantity in \eqref{mono quantity} is closely related to the one in \eqref{old mono} when $n=3$ and $\beta=2$. Later, Chodosh-Li \cite{ChodoshLi2021} applied the method to study Green functions on minimal hyper-surfaces in $\R^4$ and proved the conjecture that any complete, two-sided, stable minimal hyper-surface in $\R^4$ must be a hyper-plane. In \cite{AMO2021}, Agostiniani-Mazzieri-Oronzio also studied an analogous monotone quantity involving the Green function and gave an alternative proof to the positive mass theorem in dimension 3.

\medskip

In the PDE theory, the $p$-Laplacian $\Delta_p u:=\div\left(|\nabla u|^{p-2}\nabla u\right)$ is a natural extension of the classical Laplace operator $\Delta$. In particular,  $\Delta_2=\Delta$. It shares many fundamental properties on the spectral and function theory with the standard Laplacian $\Delta$, e.g. gradient estimates, eigenvalue estimates and comparison, existence of Green functions (\cite{Cheng1975, ChengYau1975, LiTam1992}), for instance, see \cite{Holopainen1999, KotschwarNi2009, WangZhang2011, NaberValtorta2014, SungWang2014,  MunteanuWangL2019}. On the other hand, the $p$-harmonic functions can also be used to study
the inverse mean curvature flow on $\mathbb{R}^n$ (\cite{HuiskenIlmanen1999}). Indeed, if $u$ solves the $p$-harmonic equation, then $v=(1-p)\log u$ satisfies 
\begin{equation}\label{IMCF}
\div\left(|\nabla v|^{2-p}{\nabla v}\right)=|\nabla v|^p.
\end{equation}
This is closely related to the level set formulation of inverse mean curvature flow in $\R^n$ when $p=1$ (\cite{HuiskenIlmanen1999}). If $v$ is in addition proper, the regular level sets of $v$ give a family of compact hyper-surfaces which evolve by the inverse mean curvature flow \cite{HuiskenIlmanen1999, KotschwarNi2009, MariRigoliSetti2019}. Moser \cite{Moser2007, Moser2008, Moser2015} constructed a weak solution to \eqref{IMCF} for $p=1$ by taking the limit $v:=\lim_{p\to 1^+}(1-p)\log u_p$, where $u_p$ is a positive $p$-harmonic function with suitable boundary condition (see also \cite{KotschwarNi2009, MariRigoliSetti2019}). More recently, Agostiniani-Mazzieri-Oronzio \cite{AMO2021} have outlined an approach based on studying the monotonicity of $p$-harmonic functions for $p\to 1$ to show the Riemannian Penrose inequality in dimension $3$.

Motivated by the work of Munteanu-Wang \cite{MunteanuWang2021} and the geometric applications of $p$-harmonic functions, we investigate the monotonicity of the $p$-Green functions under non-negative scalar curvature assumption. By studying the regularized $p$-harmonic functions, we establish a monotonicity formula of the $p$-Green functions. This generalizes the monotonicity formula of Green functions in \cite{MunteanuWang2021} to the general $p$-Green functions.

The main result in this work is the following sharp monotonicity formula for comparison with the Euclidean space.
\begin{thm}\label{main-Thm}
Suppose that $(M^3,g)$ is a $3$-dimensional complete non-compact Riemannian manifold such that
$R_{M}\ge 0$ and $H_2(M^3,\mathbb{Z})=0$, where $R_{M}$ denotes the scalar curvature. Let $\hat{u}(x)=G(x_{0},x)$ be the minimal positive $p$-Green function on $M$ for $ 1<p\leq 2$ and $x_{0}\in M$ such that $\hat{u}\to 0$ as $x\to +\infty$. For any regular value $t$ of $\hat{u}$, define
\begin{equation}
F(t):=\int_{\Sigma_t} |\nabla \hat{u}|^2 \; dA, \ \ \Sigma_t:=\{\hat{u}=t\}.
\end{equation}
Then we have the following monotonicity formulas:
\begin{enumerate}\setlength{\itemsep}{1mm}
\item[\bf (a)] For any regular value $t>0$.
\begin{equation*}
    F'(t)\leq 4\pi \left(\frac{3-p}{p-1}\right)^2 t+t^{-1} F(t);
\end{equation*}
\item[\bf (b)] For any regular values $t_1<t_2$,
\begin{equation*}
    t_{2}^{-1}F(t_{2})-4\pi\left(\frac{3-p}{p-1}\right)^2t_{2} \leq t_{1}^{-1}F(t_{1})-4\pi\left(\frac{3-p}{p-1}\right)^2t_{1}.
\end{equation*}
\end{enumerate}
Moreover, we also have the following rigidity results:
\begin{enumerate}\setlength{\itemsep}{1mm}
\item[\bf (a')]  If the equality in {\bf (a)} holds for some regular value $t_0>0$, then $t$ is regular for all $t>t_0$ and the superlevel set $\{x\in M: \hat{u}(x)>t_0 \}$ is isometric to the Euclidean ball;
\item[\bf (b')]  If the equality in {\bf (b)} holds for some regular values $s_{0}<t_{0}$, then $t$ is regular for all $t>s_{0}$ and the superlevel set $\{x\in M: \hat{u}(x)>s_0 \}$ is isometric to the Euclidean ball.
\end{enumerate}
\end{thm}

\begin{rem}
If $p\in (2,3)$ and $\hat u$ is a-priori smooth on $M$, then Theorem \ref{main-Thm} still holds, see Remark \ref{p 2 3}. It would be interesting to see whether Theorem \ref{main-Thm} holds without such smoothness assumption.
\end{rem}

The arguments in \cite{MunteanuWang2021, ChodoshLi2021} are based on Stern \cite{Stern2019} which requires higher regularity of the Green function $G(x_0,\cdot)$ of $\Delta$. More precisely, since $G(x_0,\cdot)\in C^{\infty}_{\loc}(M\setminus\{x_{0}\})$, then Sard's theorem shows that the set of critical values of $G(x_0,\cdot)$ has measure zero. Roughly speaking, this means the set of critical values of $G(x_0,\cdot)$ can be ignored. At the regular value, the quantity in \eqref{mono-11} behaves very well and so the crucial formula of Stern \cite{Stern2019} can be applied.

However, for $p\neq 2$, the $p$-Laplace equation becomes degenerate quasi-linear elliptic equation. In general, the optimal regularity one can expect 
is only $C^{1,\a}_{\loc}\cap W^{2,2}_{\loc}$, see \cite{Tolksdorf1984, ManfrediWeitsman1988}. The behaviour of the set of regular values is generally unclear. Since Sard's Theorem is not known to be applicable, the straight forward adoption of level set approach fails. When $p\in (1,2]$, we are able to overcome these by studying the regularized equations and establishing the almost monotonicity formulas. We believe that the regularization method in this work will be useful to study the $p$-harmonic functions in more general setting. This might be useful in the Green function approach to the Penrose inequality in three dimension outlined in \cite{AMO2021}.

Unlike the harmonic case ($p=2$), $F(t)$ might not be defined almost everywhere since the set of regular values of $\hat u$ might not be dense. In the following theorem, we give a natural extension of $F(t)$ on $(0,+\infty)$ and prove the similar monotonicity {\bf (b)} of Theorem \ref{main-Thm}.

\begin{thm}\label{generalized F Thm}
Under the assumptions in Theorem~\ref{main-Thm}, for any $t>0$, we define
\begin{equation}
F(t):=\int_{\Sigma_t\cap\{|\nabla \hat{u}|>0\}} |\nabla \hat u|^2 \; dA, \ \ \Sigma_t:=\{\hat{u}=t\}.
\end{equation}
Then the quantity
\begin{equation}\label{mono-11}
t^{-1}F(t)-4\pi\left(\frac{3-p}{p-1}\right)^2t
\end{equation}
is continuous and non-increasing with respect to $t$.
\end{thm}

By studying the regularization at infinity more carefully, we prove some comparison inequalities on $F(t)$ and the area of the level set of the $p$-Green function when $1<p\leq 2$. As pointed out by Chodosh-Li \cite[Remark 17]{ChodoshLi2021}, we can obtain a sharp upper bound of the monotonicity quantity under Ricci curvature lower bound. We also slightly generalize the rigidity result in \cite{MunteanuWang2021}.
\begin{cor}\label{main-monot-zerothorder}
Under the assumptions in Theorem~\ref{main-Thm}, if in addition $\Ric\geq -k$ for some constant $k>0$ on $M$,
then we have the following monotonicity formulas:
\begin{enumerate}\setlength{\itemsep}{1mm}
\item[\bf (c)] For any $t>0$.
\begin{equation*}\label{mono-1}
F(t)\leq 4\pi\left(\frac{3-p}{p-1}\right)^2 t^2;
\end{equation*}
\item[\bf (d)] For any regular value $t>0$,
\begin{equation*}\label{mono-2}
\mathrm{Area}\left( \{x\in M: \hat{u}(x)=t \}\right) \geq \left[4\pi \cdot \left(\frac{3-p}{p-1} \right)^2 t^2 \right]^{-\frac{p-1}{3-p}}.
\end{equation*}
\end{enumerate}
Moreover, if the equality on {\bf(c)} or {\bf (d)} holds for some regular value $t_0>0$, then $(M,g)$ is isometric to the Euclidean space. 
\end{cor}

\begin{rem}
As observed in \cite{MunteanuWang2021}, if the $p$-Green function has a better asymptotic at the pole,
$$\liminf_{t\to\infty} \left(t^{-1}F(t)-4\pi\left(\frac{3-p}{p-1}\right)^2 t \right)\geq 0, $$
then {\bf (c)} holds as a equality for all regular value $t>0$ by Theorem~\ref{main-Thm}. In this case, $(M,g)$ is isometric to Euclidean space.   
\end{rem}

This article is organized as follows. In Section $2$, we review the preliminaries of the $p$-Green functions and $p$-parabolicity. We then formulate the regularized $p$-Laplace equations and include different properties and a-priori estimates satisfied by the solutions of these equations, which approximate the $p$-Green functions. In Section $3$, we derive some almost monotonicity formulas. In Section $4$, Theorem \ref{main-Thm}, \ref{generalized F Thm} and Corollary \ref{main-monot-zerothorder} will be established by using these almost monotonicity formulas.

\medskip

{\it Acknowledgement}:
J. Chu was partially supported by Fundamental Research Funds for the Central Universities (No. 7100603592).

\section{Basic Setting}\label{sec: pre}
Let $M^n$ be an $n$-dimensional complete manifold. Let $p\in (1,n]$, a function $u\in W^{1,p}_{\loc}(M)$ is said to be $p$-harmonic if it is a weak solution to the equation
\begin{equation}
\Delta_p u := \div\left(|\nabla u|^{p-2}\nabla u\right)=0.
\end{equation}
Namely for any $\phi\in C^{\infty}_{0}(M)$,
\begin{equation}
\int_M\la \nabla \phi, \nabla u\ra |\nabla u|^{p-2} =0.
\end{equation}
In particular, $u$ is harmonic if $p=2$.  
A function $G:(M\times M)\setminus\mathrm{Diag}(M)\to \mathbb{R}$ is said to be a $p$-Green function if $-\Delta_p G(z,\cdot)=\delta_{z}$, that is to say,
\begin{equation}\label{Greendef}
\int_M |\nabla G(z,\cdot )|^{p-2}\langle\nabla G(z,\cdot ),\nabla \varphi \rangle \, d\mathrm{vol}_g = \varphi(z)
\end{equation}
for any $\varphi\in C^\infty_{0}(M)$. Here $\delta_{z}$ denotes the Dirac mass centred at $z\in M$.

\medskip

A manifold $M$ is called $p$-parabolic if it does not admit a positive $p$-Green function. One necessary condition for $M$ to admit a positive $p$-Green function is a large volume growth (\cite{Varopoulos1981, Grigoryan1983, Grigoryan1985} for $p=2$, \cite{KeselmanZorich1996} for $p\in(1,n]$), i.e.
\begin{equation}\label{Volumegrowth}
\int_{1}^{\infty} \lf( \frac{t}{V(x_0,t)} \ri) ^{\frac{1}{p-1}} dt < \infty,
\end{equation}
for some $x_0\in M$ where $V(x_0,t)$ is the volume of the geodesic ball with radius $t$ centred at $x_0$. Note that it is not a sufficient condition. As pointed out in \cite{Holopainen1999}, one can make a conformal change on $(\R^n,g_{\mathrm{Euc}})$, which is $n$-parabolic, such that \eqref{Volumegrowth} is satisfied but $p$-parabolicity is a conformal-invariant property.

\medskip

When $p=2$, under certain curvature assumptions and a volume comparison condition, one can construct a positive $G$ by compact exhaustion as in \cite{LiTam1992}. Moreover, if \eqref{Volumegrowth} is satisfied, $G(x_0,x)\to 0$ as $x\to +\infty$ for fixed $x_0\in M$. This result was generalised to $p\in (1,n]$ in \cite{Holopainen1992, Holopainen1999}. It was further shown that such construction of positive $p$-Green functions only requires that $M$ satisfies the so-called volume doubling property and a Poincar\'e type inequality. Again, if \eqref{Volumegrowth} is satisfied, $G(x_0,x)\to 0$ as $x\to +\infty$ for fixed $x_0\in M$.  Furthermore, by a comparison theorem for Green functions (\cite{Holopainen1992} for $p=n$, \cite{MariRigoliSetti2019} for $p\in(1,n]$), this $p$-Green function constructed is unique and minimal among positive solutions to \eqref{Greendef}. For example, such $G$ exists on manifolds with finite first Betti number which has non-negative Ricci curvature outside a compact set or asymptotically non-negative sectional curvature.

\medskip

Motivated by this, we will assume that the minimal $p$-Green function $G$ exists and $\hat{u}(x)=G(x_0,x)$ satisfies $\hat{u}\to 0$ as $x\to +\infty$ for fixed $x_0\in M$. As in Theorem \ref{generalized F Thm}, for any $t\in(0,+\infty)$, we define
\begin{equation}
F(t) = \int_{\Sigma_t\cap\{|\nabla\hat{u}|>0\}} |\nabla\hat{u}|^2 \; dA.
\end{equation}

Before we study the monotonicity of $F$, we first recall the asymptotic of the Green function near the pole which is modelled by the Euclidean structure.

\begin{prop}\label{asymptotic-Green-pole}\cite{KichenassamyVeron1986} \cite[Theorem 2.4]{MariRigoliSetti2019}
Let $\mu(r)=(4\pi)^{-\frac{1}{p-1}}\left(\frac{p-1}{3-p}\right)r^{-\frac{3-p}{p-1}}$ and $p\in (1,3)$. The $p$-Green function $G(x_0,x)$ satisfies the following: as $r=d(x,x_0)\to 0$,
\begin{eqnarray*}
G-\mu(r)&=& o\left(r^{-\frac{3-p}{p-1}}\right),\\[1.5mm]
\big|\nabla G-\mu'(r)\nabla r\big|&=&o\left(r^{-\frac{2}{p-1}}\right),\\
\left|\nabla ^2G-\mu''(r)dr\otimes dr-\frac{\mu'(r)}{r}\left(g-dr\otimes dr\right)\right|&=&o\left(r^{-\frac{p+1}{p-1}}\right).
\end{eqnarray*}
\end{prop}
In particular, the above proposition implies that $\nabla \hat{u}\neq 0$ and $\hat{u}$ is smooth near its pole $x_0$. In other words, if $t$ is sufficiently large, then $t$ is a regular value of $\hat{u}$. However, Sard's theorem might not be applied for small $t$ since $\hat{u}$ does not have enough regularity when $p\neq2$. Then the set of critical values might not have measure zero. To overcome this difficulty, we study the behaviour of $\hat{u}$ by considering the solution to the regularized equation.

\subsection{Regularized equation}
Recall that $p\in (1,2]$ and $\hat{u}=G(x_{0},\cdot)$ is the $p$-Green function on $M$. Let $a<b$ be two positive numbers such that $t$ is regular value of $\hat{u}$ for all $t\geq b$. We choose an open domain $D$ so that
\begin{itemize}\setlength{\itemsep}{1mm}
\item $D\Subset M\setminus \{x_0\}$ where $x_0$ is the pole of $\hat{u}$;
\item $\{a<\hat{u}<b\}\Subset D$;
\item $\partial D$ is smooth.
\end{itemize}
We consider the following perturbed equation on $D$:
\begin{equation}\label{regularized-harmonic}
\begin{cases}
\ \mathrm{div}\left(\phi_{\e}(|\nabla u_{\e}|) \cdot \nabla u_{\e} \right)=0 & \text{in $D$};\\
\ u_{\e} = \hat u & \text{on $\de D$},
\end{cases}
\end{equation}
where $\phi_{\e}(s)=(s^2+\e)^{\frac{p-2}{2}}$ so that $u_{\e}\in C^\infty(D)$ by standard regularity theory of elliptic PDE. In the following proposition, we collect some well-known facts about $u_{\e}$.

\begin{prop}\label{Regularity-approximation-stability}
The function $u_{\e}$ satisfies the following properties:
\begin{enumerate}\setlength{\itemsep}{1mm}
\item[(i)] On the region $\widetilde D \subset D$ where $|\nabla\hat{u}|\neq 0$, $u_\e$ is uniformly bounded in $C^k_{\loc}(\widetilde D)$ for all $k\in \mathbb{N}$ and $\e>0$;
\item[(ii)] There is $\alpha\in(0,1)$ such that for any $\e>0$ and $D'\Subset D$,
\[
\|u_{\e}\|_{C^{1,\a}(D')}+\|u_{\e}\|_{W^{2,2}(D')} \leq C(\hat{u},\alpha,D',D).
\]
\item[(iii)] For any $D'\Subset D$, $u_{\e}\rightarrow\hat{u}$ in $C^{1}(D')$ as $\e\rightarrow0$.
\item[(iv)] There is $\e_{0}$ such that for any $\e\in(0,\e_{0})$ and $t\in(a,b)$,
\[
\int_{\{u_{\e}=t\}\cap\{|\nabla u_{\e}|>0\}}|\nabla u_{\e}|^{p-1} \leq C(\hat{u},b).
\]
\end{enumerate}
\end{prop}

\begin{proof}
By the standard elliptic PDE theory, we obtain (i). For (ii), the $C^{1,\a}_{\loc}$ and $W^{2,2}_{\loc}$ estimates follow from \cite[Theorem 1]{Tolksdorf1984} and \cite[Lemma 2.1]{ManfrediWeitsman1988} thanks to the range of $p$. To prove (iii), for $\e\in[0,1)$ and $v-\hat{u}\in W_{0}^{1,p}(D)$ (here $W_{0}^{1,p}(D)$ be the completion of $C_{0}^{\infty}(D)$ in $W^{1,p}(D)$), define the functionals
\begin{equation}
I_{\e}(v) := \int_{D}(|\nabla v|^{2}+\e)^{\frac{p}{2}}, \quad I(v)=I_{0}(v).
\end{equation}
It is well-known that $u_{\e}$ and $\hat{u}$ are the unique minimizers of $I_{\e}$ and $I$ respectively, and
\begin{equation}
\lim_{\e\rightarrow0}I_{\e}(u_{\e}) = I(\hat{u}).
\end{equation}
The standard argument (see e.g \cite[(1.6)]{Lewis77}) shows that
\begin{equation}
\lim_{\e\rightarrow0}\|u_{\e}-\hat{u}\|_{W^{1,p}(D)} = 0.
\end{equation}
Combining this with $C^{1,\a}_{\loc}$ estimate, we obtain (iii).

We now prove (iv). Recall that $b$ is a regular value of $\hat{u}$. Let $U\Subset D \cap\{|\nabla \hat u|>0\}$ be a neighborhood of $\{\hat u=b\}$. Thanks to (i) and (iii), $u_{\e}\rightarrow\hat{u}$ in $C^{\infty}(U)$ and so
\begin{equation}
\lim_{\e\rightarrow0}\int_{\{u_{\e}=b\}}(|\nabla u_{\e}|+1)^{p-1}
= \int_{\{\hat{u}=b\}}(|\nabla\hat{u}|+1)^{p-1}.
\end{equation}
Then there is $\e_{0}$ such that for any $\e\in(0,\e_{0})$,
\begin{equation}
\int_{\{u_{\e}=b\}}(|\nabla u_{\e}|+1)^{p-1} \leq C(\hat{u},b).
\end{equation}
We split the argument into two cases.

\medskip
\noindent
{\bf Case 1.} $t$ is a regular value of $u_{\e}$.
\medskip

Integrating \eqref{regularized-harmonic} on $\{t<u_{\e}<b\}$ and using the divergence theorem,
\begin{equation}\label{gradient p-1 eqn 1}
\int_{\{u_{\e}=t\}}(|\nabla u_{\e}|^2+\e)^{\frac{p-2}{2}}|\nabla u_{\e}|
= \int_{\{u_{\e}=b\}}(|\nabla u_{\e}|^2+\e)^{\frac{p-2}{2}}|\nabla u_{\e}| \leq C(\hat{u},b),
\end{equation}
and then
\begin{equation}
\int_{\{u_{\e}=t\}\cap\{|\nabla u_{\e}|>0\}}|\nabla u_{\e}|^{p-1} = \int_{\{u_{\e}=t\}}|\nabla u_{\e}|^{p-1} \leq C(\hat{u},b).
\end{equation}

\medskip
\noindent
{\bf Case 2.} $t$ is a critical value of $u_{\e}$.
\medskip

By a similar argument of \cite[Lemma 9]{ChodoshLi2021}, the map
\begin{equation}
t\mapsto\int_{\{u_{\e}=t\}}|\nabla u_{\e}|^{p-1}
\end{equation}
is continuous. By $u_{\e}\in C^{\infty}(\Omega)$ and Sard's theorem, there is a sequence of regular values $t_{i}$ such that $t_{i}\rightarrow t$. Then Step 1 shows
\begin{equation}
\int_{\{u_{\e}=t\}\cap\{|\nabla u_{\e}|>0\}}|\nabla u_{\e}|^{p-1}
=  \lim_{i\rightarrow\infty}\int_{\{u_{\e}=t_{i}\}}|\nabla u_{\e}|^{p-1}
\leq  C(\hat{u},b).
\end{equation}
This complete the proof of (iv).

\end{proof}

\begin{lma}\label{one-comp}
There is $\e_{0}>0$ such that for any $\e\in(0,\e_{0})$ and regular value $t$ of $u_{\e}$  in $(a,b)$, the level set
\begin{equation}
\Sigma_{\e,t} := \{u_{\e}=t\}
\end{equation}
has only one connected component.
\end{lma}

\begin{proof}
We follow the argument of \cite[Theorem 1.1]{AMO2021}. Suppose that $\Sigma'$ and $\Sigma''$ are two disjoint connected components. By the triviality of $H_2(M,\mathbb{Z})$, each closed $2$-dimensional surface is the boundary of some $3$-dimensional bounded open domain. This shows $\de\Omega'=\Sigma'$ and $\de\Omega''=\Sigma''$ for some bounded open domains $\Omega'$ and $\Omega''$. We split the argument into two cases.

\medskip
\noindent
{\bf Case 1.} $\Omega'\cap\Omega''\neq\emptyset$.
\medskip

By $\Sigma'\cap\Sigma''=\emptyset$, we have $\Omega'\Subset\Omega''$ or $\Omega''\Subset\Omega'$. Without loss of generality, we may assume that $\Omega'\Subset\Omega''$. Set $\Omega'''=\Omega''\setminus\Omega'$. Then we either have $x_{0}\notin\Omega'$ or $x_{0}\notin\Omega'''$, where $x_{0}$ is the pole of the Green function. Recall that
\begin{equation}
\{a<\hat{u}<b\} \Subset D.
\end{equation}
Then there exist $\hat{a}$, $\hat{b}$ such that
\begin{equation}
0<\hat{a}<a<b<\hat{b}, \ \
\{\hat{a}<\hat{u}<\hat{b}\} \Subset D.
\end{equation}
If $x_{0}\notin\Omega'$. Since $u_{\e}=t$ on $\de\Omega'=\Sigma'$, then by Proposition \ref{Regularity-approximation-stability} (iii), there is $\e_{0}>0$ such that for any $\e\in(0,\e_{0})$,
\begin{equation}
\hat{a} < \hat{u} < \hat{b} \ \ \text{on $\Sigma'=\de\Omega'$}.
\end{equation}
By the maximum principle,
\begin{equation}
\hat{a} < \hat{u} < \hat{b} \ \ \text{in $\Omega'$}
\end{equation}
and so
\begin{equation}
\Omega' \subset \{\hat{a}<\hat{u}<\hat{b}\} \Subset D.
\end{equation}
This implies that $\Omega'$ is a subset of the domain of $u_{\e}$. Since $u_{\e}=t$ on $\Sigma'=\de\Omega'$, then the maximum principle shows that $u_{\e}$ is constant in $\Omega'$, which is impossible. If $x_{0}\notin\Omega'''$, we may derive a contradiction by the same argument.

\medskip
\noindent
{\bf Case 2.} $\Omega'\cap\Omega''=\emptyset$.
\medskip

By the similar argument of Case 1, this case will not happen either.

\end{proof}

\subsection{Properties of $F_{\e}$ and $F$}
For any $t>0$, recall that
\begin{equation}
F(t) = \int_{\{\hat{u}=t\}\cap\{|\nabla\hat{u}|>0\}}|\nabla\hat{u}|^2 \; dA.
\end{equation}
To study the behaviour of $F$. We define
\begin{equation}
F_\e(t) := \int_{\{u_{\e}=t\}\cap\{|\nabla u_{\e}|>0\}}|\nabla u_\e|^2 \; dA.
\end{equation}

\begin{lma}\label{local Lipschitz}
For each $\e$, the function $F_{\e}$ is locally Lipschitz in $(a,b)$.
\end{lma}

\begin{proof}

Let $t\in(a,b)$. For each $\e$, $u_{\e}$ is a solution of linear elliptic PDE with smooth coefficient. From \cite{HardtSimon1989, Lin1991, CNV2015}, it follows that
\begin{equation}\label{Nodal set results}
\mathcal{H}^{n-1}(\Sigma_{\e,t}) \leq C_{\e}, \ \
\mathrm{dim}_{H}(\Sigma_{\e,t}\cap\{|\nabla u_{\e}|=0\}) \leq n-2.
\end{equation}
Here $C_{\e}$ is a constant depending on $\e$.
By the same argument of \cite[Lemma 9]{ChodoshLi2021}, $F_{\e}$ is continuous. Or one may derive the continuity of $F_{\e}$ by the continuity of $F_{\e,\delta}$ and \eqref{difference F e and F e delta}. Next, we follow the argument of \cite[Lemma 10]{ChodoshLi2021} to prove that $F_{\e}$ is locally Lipschitz. For any regular values of $u_{\e}$ in $(a,b)$, $t_{1}<t_{2}$, we define
\begin{equation}
\Omega_{\e,t_{1},t_{2}}
:= \{x\in M:t_{1}<u_{\e}(x)<t_{2}\}
\end{equation}
and compute
\begin{equation}
\begin{split}
F_{\e}(t_{1})-F_{\e}(t_{2}) = {} & \lim_{\delta\rightarrow0}\int_{\Omega_{\e,t_{1},t_{2}}}
\left\langle(|\nabla u_{\e}|^{2}+\delta)^{\frac{1}{2}}\nabla u_{\e},\frac{\nabla u_{\e}}{|\nabla u_{\e}|}\right\rangle \\
= {} & \lim_{\delta\rightarrow0}\int_{\Omega_{\e,t_1,t_2}}\div\left((|\nabla u_{\e}|^{2}+\delta)^{\frac{1}{2}}\nabla u_{\e}\right).
\end{split}
\end{equation}
Using \eqref{regularized-harmonic}, we compute
\begin{equation*}
\begin{split}
& \div\left((|\nabla u_{\e}|^{2}+\delta)^{\frac{1}{2}}\nabla u_{\e}\right) \\[1mm]
= {} & \div\left((|\nabla u_{\e}|^{2}+\delta)^{\frac{1}{2}}\phi_{\e}^{-1}\phi_{\e}\nabla u_{\e}\right) \\[1.5mm]
= {} & \left\langle\phi_{\e}\nabla u_{\e},\nabla\left((|\nabla u_{\e}|^{2}+\delta)^{\frac{1}{2}}\phi_{\e}^{-1}\right)\right\rangle \\
= {} & \left(|\nabla u_{\e}|(|\nabla u_{\e}|^{2}+\delta)^{-\frac{1}{2}}
-(p-2)\frac{|\nabla u_{\e}|(|\nabla u_{\e}|^{2}+\delta)^{\frac{1}{2}}}{|\nabla u_{\e}|^{2}+\e}\right)\langle\nabla u_{\e},\nabla|\nabla u_{\e}|\rangle.
\end{split}
\end{equation*}
By co-area formula and \eqref{Nodal set results},
\begin{equation*}
|F_{\e}(t_{1})-F_{\e}(t_{2})| \leq (3-p)\int_{\Omega_{\e,t_{1},t_{2}}}|\nabla u_{\e}||\nabla^{2}u_{\e}|
\leq C_{\e}\int_{t_{1}}^{t_{2}}\mathcal{H}^{n-1}(\Sigma_{\e,t})dt \leq C_{\e}|t_{1}-t_{2}|.
\end{equation*}
Combining this with Sard's theorem and the continuity of $F_{\e}$, we prove that $F_{\e}$ is local Lipschitz.
\end{proof}

\begin{lma}\label{gradient p-1 1}
For any regular value $t$ of $\hat{u}$, we have
\begin{equation}
\int_{\Sigma_{t}}|\nabla\hat{u}|^{p-1} = 1.
\end{equation}
\end{lma}

\begin{proof}
We first show that for any regular value $t\in(a,b)$,
\begin{equation}\label{gradient p-1 constant}
\int_{\Sigma_{t}}|\nabla\hat{u}|^{p-1} = \int_{\Sigma_{b}}|\nabla\hat{u}|^{p-1}.
\end{equation}
By Proposition \ref{Regularity-approximation-stability}, when $\e$ is sufficiently small, $t$ is also regular values of $u_{\e}$. By \eqref{gradient p-1 eqn 1},
\begin{equation}\label{gradient p-1 eqn}
\int_{\{u_{\e}=t\}}(|\nabla u_{\e}|^2+\e)^{\frac{p-2}{2}}|\nabla u_{\e}|
= \int_{\{u_{\e}=b\}}(|\nabla u_{\e}|^2+\e)^{\frac{p-2}{2}}|\nabla u_{\e}|,
\end{equation}
By Proposition \ref{Regularity-approximation-stability} again, $u_{\e}\to\hat{u}$ in $C^{\infty}(U)$ for some neighborhood $U$ of $\Sigma_{t}\cup\Sigma_{b}$. Letting $\e\to0$ in \eqref{gradient p-1 eqn}, we obtain \eqref{gradient p-1 constant}. Since $a$ and $b$ are arbitrary, then Proposition \ref{asymptotic-Green-pole} shows
\begin{equation}
\int_{\Sigma_{t}}|\nabla\hat{u}|^{p-1} = \lim_{b\to+\infty}\int_{\Sigma_{b}}|\nabla\hat{u}|^{p-1}=1
\end{equation}
for any regular value $t>0$.
\end{proof}

{
\begin{lma}\label{integral approximation}
For any $t>0$ and $0<\delta'<\delta<\delta''$, we have
\begin{equation}\label{integral inequality}
\int_{\{\hat{u}=t\}\cap\{|\nabla\hat{u}|>\delta''\}}|\nabla\hat{u}|^{2}+o(1)
\leq \int_{\{u_{\e}=t\}\cap\{|\nabla u_{\e}|>\delta\}}|\nabla u_{\e}|^{2}
\leq \int_{\{\hat{u}=t\}\cap\{|\nabla\hat{u}|>\delta'\}}|\nabla\hat{u}|^{2}+o(1).
\end{equation}
where $o(1)$ denotes a term satisfying $\lim_{\e\to0}o(1)=0$. In particular, for each $t$, there is a sequence $\delta_i\to0$ such that
\begin{equation}\label{integral limit}
\lim_{\e\to0}\int_{\{u_{\e}=t\}\cap\{|\nabla u_{\e}|>\delta_i\}}|\nabla u_{\e}|^{2}
= \int_{\{\hat{u}=t\}\cap\{|\nabla\hat{u}|>\delta_i\}}|\nabla\hat{u}|^{2}.
\end{equation}
\end{lma}

\begin{proof}
We first show how to use \eqref{integral inequality} to derive \eqref{integral limit}. The function $|\nabla\hat{u}|$ is smooth on the open manifold $\{\hat{u}=t\}\cap\{|\nabla\hat{u}|>0\}$. Thanks to Sard's theorem, there is a sequence of regular values $\delta_{i}\to0$ so that for all $i\in \mathbb{N}$,
\begin{equation}
\int_{\{\hat{u}=t\}\cap\{|\nabla\hat{u}|>\delta_i\}}|\nabla\hat{u}|^{2}
= \int_{\{\hat{u}=t\}\cap\{|\nabla\hat{u}|\geq\delta_i\}}|\nabla\hat{u}|^{2}.
\end{equation}
By letting $\e\to0$ and followed by $\delta',\delta''\to \delta_i$ in \eqref{integral inequality}, we obtain \eqref{integral limit}.

We next show \eqref{integral inequality}. For notational convenience, we introduce the following notations:
\begin{equation}
\Sigma_{t}^{\delta} := \{\hat{u}=t\}\cap\{|\nabla\hat{u}|>\delta\}, \ \
\Sigma_{\e,t}^{\delta} := \{u_{\e}=t\}\cap\{|\nabla u_{\e}|>\delta\}.
\end{equation}
By a contradiction argument, when $\e$ is sufficiently small, we see that $\Sigma_{\e,t}^{\delta}=\emptyset$ if $\Sigma_{t}^{\delta}=\emptyset$. Then \eqref{integral inequality} holds trivially. Next we may assume that $\Sigma_{t}^{\delta}\neq\emptyset$. Choose $\delta'$ and $\delta''$ such that $\delta/2<\delta'<\delta<\delta''$. For any $p\in\Sigma_{\e,t}^{\delta}$, we have
\begin{equation}
u_{\e}(p) = t, \ \ |\nabla u_{\e}|(p) > \delta.
\end{equation}
Since $u_{\e}$ converges to $\hat{u}$ in the $C^{1}$ sense, then
\begin{equation}
\hat{u}(p) = t+o(1), \ \ |\nabla \hat{u}|(p) > \delta+o(1),
\end{equation}
and so there exists a point $p'\in M$ such that
\begin{equation}
\hat{u}(p') = t, \ \ |\nabla \hat{u}|(p) > \delta', \ \ d(p,p') = o(1).
\end{equation}
It follows that
\begin{equation}\label{small distance}
d(p,\Sigma_{t}^{\delta'}) = o(1).
\end{equation}
Hence, when $\e$ is sufficiently small, the projection map $\pi:\Sigma_{\e,t}^{\delta}\to\Sigma_{t}^{\delta'}$ is well defined. Namely, for any $p\in\Sigma_{\e,t}^{\delta}$, $\pi(p)$ is the point in $\Sigma_{t}^{\delta'}$ such that
\begin{equation}
d(p,\pi(p)) = d(p,\Sigma_{t}^{\delta'}).
\end{equation}
We split the proof of \eqref{integral inequality} into three steps.

\medskip
\noindent
{\bf Step 1.} The map $\pi:\Sigma_{\e,t}^{\delta}\to\Sigma_{t}^{\delta'}$ is injective.
\medskip

Suppose that $\pi(p_{1})=\pi(p_{2})$ for some $p_{1}\neq p_{2}\in\Sigma_{\e,t}^{\delta}$. Write $q=\pi(p_{1})=\pi(p_{2})$. Let $\gamma(t)$ be the geodesic through $q$ such that
\begin{equation}
\gamma(0) = q, \ \ \gamma'(0) = \frac{\nabla\hat{u}(q)}{|\nabla\hat{u}|(q)}.
\end{equation}
Since $\pi$ is the projection, then $p_{1}$ and $p_{2}$ must lie on the geodesic $\gamma$, i.e. there exist $s_{1}\neq s_{2}$ such that
\begin{equation}
p_{1} = \gamma(s_{1}), \ \ p_{2} = \gamma(s_{2}).
\end{equation}
Define $f_\e (s)=u_{\e}(\gamma(s))$. Then by the $C^1$ convergence of $u_{\e}$ to $\hat{u}$,
\begin{equation}\label{f derivative at 0}
f'_\e(0) = \langle\nabla u_\e(q),\gamma'(0)\rangle = |\nabla\hat{u}|(q)+o(1) > \delta'+o(1) > \frac{\delta}{2}.
\end{equation}

By $p_{1},p_{2}\in\Sigma_{\e,t}^{\delta}$, we obtain $f_\e(s_{1})=f_\e(s_{2})=t$.
Then the mean value theorem shows that $f'_\e(s_{0})=0$ for some $s_{0}$ between $s_{1}$ and $s_{2}$. Thanks to \eqref{small distance}, we see that $s_{1}=o(1)$, $s_{2}=o(1)$, and so $s_{0}=o(1)$. Furthermore by the uniform $C^{1,\a}$ estimate of $u_\e$ near $\Sigma_{t}^{\delta'}$, this contradicts with \eqref{f derivative at 0} when $\e$ is sufficiently small.

\medskip
\noindent
{\bf Step 2.} $\Sigma_{t}^{\delta''}\subset\pi(\Sigma_{\e,t}^{\delta})$.
\medskip

For any $q\in\Sigma_{t}^{\delta''}$, let $\gamma(t)$ be the geodesic through $q$ such that
\begin{equation}
\gamma(0) = q, \ \ \gamma'(0) = \frac{\nabla\hat{u}(q)}{|\nabla\hat{u}|(q)}.
\end{equation}
Write $f_{\e}(s)=u_{\e}(\gamma(s))$. The similar calculation of Step 1 shows 
\begin{equation}
f_{\e}'(0) = \langle\nabla u_{\e}(q),\gamma'(0)\rangle = |\nabla\hat{u}|(q)+o(1) > \delta''+o(1).
\end{equation}
Combining this with $u\in C^{1,\alpha}$, we see that $\gamma$ must interest with $\Sigma_{\e,t}^{\delta}$. The intersection point is the preimage of $q$.

\medskip
\noindent
{\bf Step 3.} Prove \eqref{integral inequality}.
\medskip

Step 1 shows that the map $\pi:\Sigma_{\e,t}^{\delta}\to\pi(\Sigma_{\e,t}^{\delta})$ is one-to-one. Combining this with the convergence of $u_{\e}$, 
\begin{equation}
\int_{\Sigma_{\e,t}^{\delta}}|\nabla u_{\e}|^{2}
= \int_{\pi(\Sigma_{\e,t}^{\delta})}\pi_{*}(|\nabla u_{\e}|^{2})
= \int_{\pi(\Sigma_{\e,t}^{\delta})}|\nabla\hat{u}|^{2}+o(1).
\end{equation}
Using Step 2, we see that $\Sigma_{t}^{\delta''}\subset\pi(\Sigma_{\e,t}^{\delta})\subset\Sigma_{t}^{\delta'}$ and so
\begin{equation}
\int_{\Sigma_{t}^{\delta''}}|\nabla\hat{u}|^{2}+o(1)
\leq \int_{\Sigma_{\e,t}^{\delta}}|\nabla u_{\e}|^{2}
\leq \int_{\Sigma_{t}^{\delta'}}|\nabla\hat{u}|^{2}+o(1).
\end{equation}
This completes the proof. 
\end{proof}
}

\begin{lma}\label{convergence F e and F}
The function $F$ is bounded and continuous in $(a,b)$. For any $t\in(a,b)$,
\begin{equation}
\lim_{\e\rightarrow0}F_{\e}(t) = F(t).
\end{equation}
\end{lma}

\begin{proof}
By Proposition \ref{Regularity-approximation-stability} (i) and (iii), it is clear that
\begin{equation}\label{bound gradient p-1}
\int_{\{\hat{u}=t\}\cap\{|\nabla \hat{u}|>0\}}|\nabla\hat{u}|^{p-1}
= \lim_{\delta\rightarrow0}\lim_{\e\rightarrow0}\int_{\{u_{\e}=t\}\cap\{|\nabla u_{\e}|>\delta\}}|\nabla u_\e|^{p-1} \leq C(\hat{u},b)
\end{equation}
and then
\begin{equation}
F(t) \leq \left(\sup_{\{a\leq\hat{u}\leq b\}}|\nabla\hat{u}|^{3-p}\right)\int_{\{\hat{u}=t\}\cap\{|\nabla\hat{u}|>0\}}|\nabla\hat{u}|^{p-1} \leq C(\hat{u},a,b).
\end{equation}
This shows the boundedness of $F$.

For the continuity of $F$, let $\delta >0$, we decompose $F$ into two parts:
\begin{equation}
F(t) = \int_{\{\hat{u}=t\}\cap\{|\nabla\hat{u}|>\delta\}}|\nabla\hat{u}|^{2}+\int_{\{\hat{u}=t\}\cap\{0<|\nabla\hat{u}|\leq\delta\}}|\nabla\hat{u}|^{2}.
\end{equation}
The first part is continuous with respect to $t$ while the second part satisfies
\begin{equation}\label{nabla u leq delta}
\int_{\{\hat{u}=t\}\cap\{0<|\nabla\hat{u}|\leq\delta\}}|\nabla\hat{u}|^{2}
\leq \delta^{3-p}\int_{\{\hat{u}=t\}\cap\{|\nabla\hat{u}|>0\}}|\nabla\hat{u}|^{p-1} \leq C(\hat{u},b)\delta^{3-p}.
\end{equation}
These imply that $F$ is continuous.

For the convergence, we define
\begin{equation}
F_{\e,\delta}(t) = \int_{\{u_{\e}=t\}\cap\{|\nabla u_{\e}|>\delta\}}|\nabla u_{\e}|^{2}.
\end{equation}
By Proposition \ref{Regularity-approximation-stability} (iv), we estimate the difference
\begin{equation}\label{difference F e and F e delta}
\begin{split}
|F_{\e}(t)-F_{\e,\delta}(t)|
= {} & \int_{\{u_{\e}=t\}\cap\{0<|\nabla u_{\e}|\leq\delta\}}|\nabla u_{\e}|^{2} \\
\leq {} & \delta^{3-p}\int_{\{u_{\e}=t\}\cap\{|\nabla u_{\e}|>0\}}|\nabla u_{\e}|^{p-1} \\[1.5mm]
\leq {} & C(\hat{u},b)\delta^{3-p}.
\end{split}
\end{equation}
Combining this with \eqref{nabla u leq delta}  and Lemma \ref{integral approximation},
\begin{equation}
|F_{\e}(t)-F(t)|
\leq C(\hat{u},b)\delta_{i}^{3-p}+\left|\int_{\{u_{\e}=t\}\cap\{|\nabla u_{\e}|>\delta_{i}\}}|\nabla u_{\e}|^{2}
-\int_{\{\hat{u}=t\}\cap\{|\nabla\hat{u}|>\delta_{i}\}}|\nabla\hat{u}|^{2}\right|.
\end{equation}
Letting $\e\rightarrow0$ and followed by $\delta_{i}\rightarrow0$, we obtain $\lim_{\e\rightarrow0}F_{\e}(t)=F(t)$.
\end{proof}

\section{monotonicity from the regularized equation}\label{Sec:regularized-monot}

In this section, we will obtain (almost) monotonicity of $F_\e$ using the idea of \cite{MunteanuWang2021}. We will proceed using the notations and set-up in Section~\ref{sec: pre}.

When $|\nabla u_{\e}|\neq 0$, the regularized equation \eqref{regularized-harmonic} can be written as
\begin{equation}\label{regularozed-harmonic-2}
\Delta u_{\e}=-\frac{\phi_{\e}'(|\Na u_{\e}|)}{\phi_{\e}(|\Na u_{\e}|)}\cdot\langle\nabla|\nabla u_{\e}|,\nabla u_{\e} \rangle.
\end{equation}
Write $\nu=\frac{\nabla u_{\e}}{|\nabla u_{\e}|}$, then on the regular level set of $u_{\e}$, 
\begin{equation}
\Delta u_{\e}=-\eta_{\e}(|\nabla u_{\e}|) \cdot(u_{\e})_{\nu\nu}
\end{equation}
with $\eta_{\e}(s)=(\log \phi_{\e}(s))'\cdot s$. It is clear that
\begin{equation}\label{choice-eta}
\eta_{\e}(s)=\frac{(p-2)s^2}{s^2+\e},\quad \; \eta_{\e}'(s)=\frac{2(p-2)\e s}{(s^2+\e)^2}.
\end{equation}
In particular, when $\e=0$, $\phi(s)=s^{p-2}$ and $\eta(s)=p-2$, \eqref{regularozed-harmonic-2} corresponds to the original $p$-harmonic equation.

\subsection{Improved Kato inequality}
We start with the improved Kato inequality which gives sharp lower bound of the Hessian compared to the Euclidean $p$-Green function.
\begin{lma}\label{improved-Kato}
At $q\in M$ where $|\nabla u_{\e}|\neq0$, we have
\begin{equation}
|\nabla^2 u_{\e}|^2\ge \min \left(\frac{\eta_{\e}^2+2\eta_{\e}+3}{2},2 \right)|\nabla|\nabla u_{\e}||^2.
\end{equation}
and
\begin{equation}
|\nabla^2 u_{\e}|^2 \geq \left(\frac{\eta_{\e}^2+2\eta_{\e}+3}{2}\right)\langle\nabla|\nabla u_{\e}|,\nu \rangle^{2}.
\end{equation}
\end{lma}

\begin{proof}
We will omit the index $\e$ for notation convenience. Choose $\{e_j\}_{j=1}^3$ such that $e_1=\nu=\frac{\nabla u}{|\nabla u|}$. Near $q$, $u$ is smooth by the standard elliptic PDE theory. Then we have
\begin{equation}
\nabla_j |\nabla u|=\frac{\nabla^2 u(e_j,\nabla u)}{|\nabla u|}=u_{1j}.
\end{equation}
This implies
\begin{equation}
|\nabla|\nabla u||^2=|u_{11}|^2+|u_{12}|^2+|u_{13}|^2.
\end{equation}
It follows from $\Delta u=-\eta\cdot\langle\nabla|\nabla u|,\nu\rangle$ that
\begin{equation}
u_{22}+u_{33}=-(\eta+1)u_{11}.
\end{equation}
By the Cauchy-Schwarz inequality,
\begin{equation}
\begin{split}
|\nabla^{2}u|^{2} \geq {} & |u_{11}|^{2}+|u_{22}|^{2}+|u_{33}|^{2}+2|u_{12}|^{2}+2|u_{13}|^{2} \\[1mm]
\geq {} & |u_{11}|^{2}+\frac{1}{2}|u_{22}+u_{33}|^{2}+2|u_{12}|^{2}+2|u_{13}|^{2} \\
= {} & |u_{11}|^{2}+\frac{(\eta+1)^{2}}{2}u_{11}^{2}+2|u_{12}|^{2}+2|u_{13}|^{2} \\[1mm]
= {} & \left(\frac{\eta^2+2\eta+3}{2}\right)|u_{11}|^{2}+2|u_{12}|^{2}+2|u_{13}|^{2} \\
\geq {} & \min \left(\frac{\eta^2+2\eta+3}{2},2 \right)|\nabla|\nabla u||^2
\end{split}
\end{equation}
and
\begin{equation}
\begin{split}
|\nabla^{2}u|^{2} \geq {} & |u_{11}|^{2}+|u_{22}|^{2}+|u_{33}|^{2} \\[1mm]
\geq {} & |u_{11}|^{2}+\frac{1}{2}|u_{22}+u_{33}|^{2} \\
= {} & |u_{11}|^{2}+\frac{(\eta+1)^{2}}{2}u_{11}^{2} \\[1mm]
= {} & \left(\frac{\eta^2+2\eta+3}{2}\right)|u_{11}|^{2} \\
= {} & \left(\frac{\eta^2+2\eta+3}{2}\right)\langle\nabla|\nabla u|,\nu \rangle^{2}.
\end{split}
\end{equation}
This completes the proof.
\end{proof}

\subsection{Almost monotonicity}
In this sub-section, we will establish the almost monotonicity on the regularized equation. 
\begin{thm}\label{monotonicity-regularized-THM}
Suppose that $p\in (1,2]$, $(5-p)+(3p-7)\b \geq 0$ and $\lambda$ satisfies
\begin{equation}
\frac{\lambda(5-p)(\b+1)}{2(3-p)}\geq \b(\b+1)+\frac{\lambda^2 (5-p)}{4(3-p)}.
\end{equation}
For regular values $t_{1}<t_{2}$ of $u_{\e}$, the function
\begin{equation}
\mathcal{H}_\e(t)= t^{-\b} F'_\e(t)+\left(\b-\frac{\lambda(5-p)}{2(3-p)}\right)t^{-\b-1}F_\e(t) -\frac{4\pi(3-p)}{(p-1)(\b-1)}t^{-\b+1}
\end{equation}
satisfies
\begin{equation}
\mathcal{H}_\e(t_{1})-\mathcal{H}_\e(t_{2})\leq \mathbf{E}_{\e,t_{1},t_{2}}+C\sqrt{\e}
\end{equation}
for some constant $C$ independent of $\e$, where
\begin{equation}
\mathbf{E}_{\e, t_{1}, t_{2}} =
\left(t^{-\b}\int_{\Sigma_t} \frac{p-2-\eta}{3-p}\cdot\frac{\eta-1}{\eta }\cdot\Delta u_{\e} \right)\bigg|_{t_{2}}^{t_{1}}.
\end{equation}
\end{thm}

\begin{proof}
We will omit the index $\e$ for convenience. Near the regular value, the family $t\mapsto \Sigma_t$ has normal velocity $\partial_t=|\nabla u|^{-1}\nu$ where $\nu=\frac{\nabla u}{|\nabla u|}$. 

The mean curvature $H$ on $\Sigma_t$ is given by
\begin{equation}
\begin{split}
H = {} & |\nabla u|^{-1}(\Delta u-\nabla^{2}u(\nu,\nu)) \\[1mm]
= {} & -\phi^{-1}\phi'\langle \nabla|\nabla u|,\nu \rangle-|\nabla u|^{-1}\langle \nabla|\nabla u|,\nu \rangle \\[1mm]
= {} & -(\phi^{-1}\phi'+|\nabla u|^{-1})\langle \nabla|\nabla u|,\nu \rangle,
\end{split}
\end{equation}
so that
\begin{equation}
\begin{split}
F'(t)&=\int_{\Sigma_t} |\nabla u|^{-1} \langle  \nabla |\nabla u|^2,\nu\rangle+  H |\nabla u|\\
&=\int_{\Sigma_t} (1-\eta) \langle \nabla|\nabla u|,\nu\rangle.
\end{split}
\end{equation}

To apply the divergence theorem, we will follow the treatment of critical values as in \cite{ChodoshLi2021}. Define
\begin{equation*}
v_\delta=\sqrt{|\nabla u|^2+\delta}
\end{equation*}
so that if $t$ is regular value of $u$, we have
\begin{equation}\label{time-der}
t^{-\b}F'(t) = \int_{\Sigma_t} (1-\eta)u^{-\b} \langle \nabla |\nabla u|,\nu\rangle
= \lim_{\delta\to 0} \int_{\Sigma_t} (1-\eta)u^{-\b} \langle \nabla v_\delta,\nu\rangle.
\end{equation}
Since $t_1<t_2$ are regular values of $u$, then
\begin{equation}
\Omega_{t_1,t_2}=\{x\in \Omega: t_1\leq u(x)\leq t_2\}
\end{equation}
is a domain with smooth boundary. We divide the argument into four steps.

\medskip
\noindent
{\bf Step 1.} Upper bound of $G(t_1)-G(t_2)$.
\medskip

Define
\begin{equation}
G(t)=(3-p)^{-1}t^{-\b} F'(t).
\end{equation}
Then the divergence theorem shows
\begin{equation}\label{G A+B}
\begin{split}
& G(t_1)-G(t_2) \\
= {} & \lim_{\delta\to 0} \left(\int_{\Sigma_t}\frac{1-\eta}{3-p} u^{-\b} \langle\nabla v_\delta,\nu \rangle  \right)\bigg|_{t_2}^{t_1}\\
= {} & \lim_{\delta\to0}\left(-\int_{\Omega_{t_1,t_2}} \frac{1-\eta}{3-p} u^{-\b} \Delta v_\delta
-\int_{\Omega_{t_1,t_2}}\left\langle \nabla v_\delta,\nabla\left(\frac{1-\eta}{3-p}u^{-\b} \right)\right\rangle\right) \\[2mm]
= {} & \lim_{\delta\to 0}\left(\mathbf{A}+\mathbf{B}\right).
\end{split}
\end{equation}
We now estimate integrals $\mathbf{A}$ and $\mathbf{B}$ separately. The Bochner formula implies
\begin{equation*}
\begin{split}
\Delta v_{\delta} = {} & \frac{1}{2}v_{\delta}^{-1}\Delta|\nabla u|^{2}-\frac{1}{4}v_{\delta}^{-3}|\nabla|\nabla u|^{2}|^{2} \\
= {} & v_{\delta}^{-1}\left(|\nabla^{2}u|^{2}-\frac{1}{4}v_{\delta}^{-2}|\nabla|\nabla u|^{2}|^{2}\right)
+v_{\delta}^{-1}\langle\nabla\Delta u,\nabla u\rangle+v_{\delta}^{-1}\Ric_{M}(\nabla u,\nabla u),
\end{split}
\end{equation*}
and so
\begin{equation}\label{integral A}
\begin{split}
\mathbf{A}
= {} & \int_{\Omega_{t_1,t_2}} -\frac{1-\eta}{3-p}u^{-\b}  v_\delta^{-1} \left(|\nabla^{2}u|^{2}-\frac{1}{4}v_{\delta}^{-2}|\nabla|\nabla u|^{2}|^{2}\right)\\
& +\int_{\Omega_{t_1,t_2}}-\frac{1-\eta}{3-p}u^{-\b}  v_\delta^{-1} \langle\nabla\Delta u,\nabla u\rangle\\
& +\int_{\Omega_{t_1,t_2}}-\frac{1-\eta}{3-p}u^{-\b}  v_\delta^{-1}\Ric_{M}(\nabla u,\nabla u).
\end{split}
\end{equation}
We now focus on the second term in $\mathbf{A}$. By the divergence theorem and \eqref{regularized-harmonic}, we have
\begin{equation}\label{second term of A}
\begin{split}
&\quad \int_{\Omega_{t_1,t_2}}-\frac{1-\eta}{3-p}u^{-\b} v_\delta^{-1}\langle\nabla\Delta u,\nabla u\rangle \\
&=\left(\int_{\Sigma_t}\frac{1-\eta}{3-p}u^{-\b}v_\delta^{-1}\Delta u \cdot |\nabla u|\right)\bigg|_{t_2}^{t_1}+\int_{\Omega_{t_1,t_2}}  \Delta u\cdot \mathrm{div}\left(\frac{1-\eta}{3-p}u^{-\b} v_\delta^{-1} \nabla u\right)\\
&=\int_{\Sigma_{t_1}}\frac{1-\eta}{3-p}t_{1}^{-\b}v_\delta^{-1}\Delta u \cdot  |\nabla u|
-\int_{\Sigma_{t_2}}\frac{1-\eta}{3-p}t_{2}^{-\b}v_\delta^{-1}\Delta u \cdot  |\nabla u|\\
&\quad +\int_{\Omega_{t_1,t_2}}  \Delta u \left\langle\phi\nabla u, \nabla\left( \frac{1-\eta}{3-p}u^{-\b} v_\delta^{-1} \phi^{-1} \right)\right\rangle.
\end{split}
\end{equation}
For the last term in \eqref{second term of A}, we compute
\begin{equation}
\begin{split}
& \int_{\Omega_{t_1,t_2}}  \Delta u \left\langle\phi\nabla u, \nabla\left( \frac{1-\eta}{3-p}u^{-\b} v_\delta^{-1} \phi^{-1} \right)\right\rangle \\
= {} & -\int_{\Omega_{t_1,t_2}}\frac{\eta'}{3-p}\cdot \Delta u\cdot u^{-\beta}\cdot v_{\delta}^{-1}\cdot\langle\nabla|\nabla u|,\nabla u\rangle \\
& -\int_{\Omega_{t_1,t_2}}\beta\frac{1-\eta}{3-p}\cdot  \Delta u\cdot v_{\delta}^{-1}\cdot u^{-\beta-1}\cdot|\nabla u|^{2} \\
& -\int_{\Omega_{t_1,t_2}}\frac{1-\eta}{3-p}\cdot  \Delta u\cdot u^{-\beta}\cdot v_{\delta}^{-3}\cdot|\nabla u|\cdot\langle\nabla|\nabla u|,\nabla u\rangle \\
& -\int_{\Omega_{t_1,t_2}}\frac{1-\eta}{3-p}\cdot  \Delta u\cdot u^{-\beta}\cdot v_{\delta}^{-1}\cdot\eta\cdot\left\langle\nabla|\nabla u|,\frac{\nabla u}{|\nabla u|}\right\rangle.
\end{split}
\end{equation}
For the last term (Ricci term) in $\mathbf{A}$,
\begin{equation*}
\begin{split}
& -\int_{\Omega_{t_1,t_2}}\frac{1-\eta}{3-p} u^{-\b} v_\delta^{-1} \Ric_{M}(\nabla u,\nabla u)\\
&=-\int_{\Omega_{t_1,t_2}}u^{-\b} v_\delta^{-1} \Ric_{M}(\nabla u,\nabla u)
+\int_{\Omega_{t,\tau}}\frac{\eta+2-p}{3-p} u^{-\b} v_\delta^{-1}\Ric_{M}(\nabla u,\nabla u).
\end{split}
\end{equation*}
Substituting the above into \eqref{integral A},
\begin{equation}\label{A}
\begin{split}
\mathbf{A}&=-\int_{\Omega_{t_1,t_2}}\frac{1-\eta}{3-p}u^{-\b} v_\delta^{-1} \left(|\nabla^{2}u|^{2}-\frac{1}{4}v_{\delta}^{-2}|\nabla|\nabla u|^{2}|^{2}\right)\\
&\quad +\int_{\Sigma_{t_1}}\frac{1-\eta}{3-p}t_{1}^{-\b}v_\delta^{-1}\Delta u \cdot  |\nabla u|
-\int_{\Sigma_{t_2}}\frac{1-\eta}{3-p}t_{2}^{-\b}v_\delta^{-1}\Delta u \cdot  |\nabla u|\\
&\quad -\int_{\Omega_{t_1,t_2}}\frac{\eta'}{3-p}\cdot \Delta u\cdot u^{-\beta}\cdot v_{\delta}^{-1}\cdot\langle\nabla|\nabla u|,\nabla u\rangle \\
& \quad -\beta\int_{\Omega_{t_1,t_2}}\frac{1-\eta}{3-p}\cdot  \Delta u\cdot u^{-\beta-1}\cdot v_{\delta}^{-1}\cdot|\nabla u|^{2} \\
& \quad -\int_{\Omega_{t_1,t_2}}\frac{1-\eta}{3-p}\cdot  \Delta u\cdot u^{-\beta}\cdot v_{\delta}^{-3}\cdot|\nabla u|\cdot\langle\nabla|\nabla u|,\nabla u\rangle \\
& \quad -\int_{\Omega_{t_1,t_2}}\frac{1-\eta}{3-p}\cdot  \Delta u\cdot u^{-\beta}\cdot v_{\delta}^{-1}\cdot\eta\cdot\left\langle\nabla|\nabla u|,\frac{\nabla u}{|\nabla u|}\right\rangle\\
&\quad -\int_{\Omega_{t_1,t_2}}u^{-\b} v_\delta^{-1}  \Ric_{M}(\nabla u,\nabla u)
+\int_{\Omega_{t_1,t_2}}\frac{\eta+2-p}{3-p} u^{-\b} v_\delta^{-1}\Ric_{M}(\nabla u,\nabla u).
\end{split}
\end{equation}
For the first term in $\mathbf{A}$, thanks to \eqref{choice-eta} and classical Kato inequality,
\begin{equation}
-\frac{1-\eta}{3-p} u^{-\b} v_\delta^{-1} \left(|\nabla^{2}u|^{2}-\frac{1}{4}v_{\delta}^{-2}|\nabla|\nabla u|^{2}|^{2}\right) \leq 0.
\end{equation}
Let $\mathcal{U}$ be the set of all regular values of $u$. Using the co-area formula, we obtain a upper bound of the first term in $\mathbf{A}$:
\begin{equation}\label{A 1}
\begin{split}
& -\int_{\Omega_{t_1,t_2}}\frac{1-\eta}{3-p}u^{-\b} v_\delta^{-1} \left(|\nabla^{2}u|^{2}-\frac{1}{4}v_{\delta}^{-2}|\nabla|\nabla u|^{2}|^{2}\right)\\
\leq {} & -\int_{(t_1,t_2)\cap\mathcal{U}} \int_{\Sigma_t}\frac{1-\eta}{3-p} u^{-\b} v_\delta^{-1}|\nabla u|^{-1}
\left(|\nabla^{2}u|^{2}-\frac{1}{4}v_{\delta}^{-2}|\nabla|\nabla u|^{2}|^{2}\right).
\end{split}
\end{equation}
On the other hand, we compute the integral $\mathbf{B}$:
\begin{equation}\label{B}
\begin{split}
\mathbf{B}&=-\int_{\Omega_{t_1,t_2}}\left\langle \nabla v_\delta,\nabla\left(\frac{1-\eta}{3-p}u^{-\b} \right)\right\rangle\\
&=   -\int_{\Omega_{t_1,t_2}}v_{\delta}^{-1}\cdot|\nabla u|\cdot\left\langle\nabla |\nabla u|,\nabla\left(\frac{1-\eta}{3-p}\cdot u^{-\beta}\right)\right\rangle \\
&=  +\int_{\Omega_{t_1,t_2}}v_{\delta}^{-1}\cdot|\nabla u|\cdot\left\langle\nabla |\nabla u|,\frac{\eta'}{3-p}\cdot u^{-\beta}\cdot\nabla|\nabla u|\right\rangle \\
& \quad -\int_{\Omega_{t_1,t_2}}v_{\delta}^{-1}\cdot|\nabla u|\cdot\left\langle\nabla |\nabla u|,\frac{1-\eta}{3-p}
\cdot(-\beta)\cdot u^{-\beta-1}\cdot\nabla u\right\rangle \\
&=  \int_{\Omega_{t_1,t_2}}\frac{\eta'}{3-p}\cdot u^{-\beta}\cdot v_{\delta}^{-1}\cdot|\nabla u|\cdot\big|\nabla|\nabla u|\big|^{2} \\
& \quad +\beta\int_{\Omega_{t_1,t_2}}\frac{1-\eta}{3-p}\cdot u^{-\beta-1}\cdot v_{\delta}^{-1}\cdot|\nabla u|\cdot\left\langle\nabla |\nabla u|,\nabla u\right\rangle.
\end{split}
\end{equation}
Substituting \eqref{A}, \eqref{A 1} and \eqref{B} into \eqref{G A+B} and using Lebesgue's dominated convergence theorem, we may let $\delta\to 0$ to obtain
\begin{equation}\label{A+B}
\begin{split}
&\quad G(t_1)-G(t_2)= \lim_{\delta\to 0} (\mathbf{A}+\mathbf{B})\\
&\leq -\int_{(t_1,t_2)\cap\mathcal{U}} \int_{\Sigma_t} \frac{1-\eta}{3-p} u^{-\b}|\nabla u|^{-2} \left(|\nabla^2 u|^2-|\nabla |\nabla u||^2 \right)\\
&\quad +\int_{\Sigma_{t_1}}\frac{1-\eta}{3-p}t_1^{-\b}\Delta u -\int_{\Sigma_\tau}\frac{1-\eta}{3-p}t_{2}^{-\b}\Delta u \\
&\quad -\int_{\Omega_{t_1,t_2}}u^{-\b} |\nabla u|^{-1}  \Ric(\nabla u,\nabla u)
+\int_{\Omega_{t,\tau}}\frac{\eta-p+2}{3-p} u^{-\b} |\nabla u|^{-1} \Ric(\nabla u,\nabla u)\\
&\quad +\b \int_{\Omega_{t_1,t_2}} \frac{1-\eta}{3-p} u^{-\b-1}
\langle\nabla|\nabla u|,\nabla u \rangle-\b\int_{\Omega_{t,\tau}} \frac{1-\eta}{3-p} \Delta u \cdot u^{-\b-1} |\nabla u|\\
&\quad +\int_{\Omega_{t_1,t_2}} \left(\frac{-\eta' |\nabla u|}{3-p}-\frac{1-\eta}{3-p} -\frac{(1-\eta)\eta}{3-p}\right) u^{-\b}\Delta u \cdot |\nabla u|^{-2}\left\langle\nabla |\nabla u|,\nabla u \right\rangle \\
&\quad +\int_{\Omega_{t_1,t_2}} \frac{\eta'}{3-p} u^{-\b} |\nabla |\nabla u||^2.
\end{split}
\end{equation}

\medskip
\noindent
{\bf Step 2.} Estimate each term in \eqref{A+B}.
\medskip

For the boundary terms in \eqref{A+B}, using \eqref{regularozed-harmonic-2} and \eqref{time-der},
\begin{equation}\label{mono-equ-1}
\begin{split}
\int_{\Sigma_{t_{1}}} \frac{1-\eta}{3-p} t_{1}^{-\b} \Delta u
&=-t_{1}^{-\b} \int_{\Sigma_{t_{1}}} \frac{1-\eta}{3-p} \cdot \eta \cdot  \langle \nabla |\nabla u|,\nu \rangle\\
&=-t_{1}^{-\b}\cdot \frac{p-2}{3-p} \int_{\Sigma_t} (1-\eta) \langle \nabla |\nabla u|,\nu \rangle\\
&\quad -t_{1}^{-\b} \int_{\Sigma_{t_{1}}} \frac{(1-\eta)(\eta-p+2)}{3-p} \langle \nabla|\nabla u|,\nu\rangle\\
&=-t_{1}^{-\b}\cdot \frac{p-2}{3-p} F'(t_1) -t_{1}^{-\b} \int_{\Sigma_t} \frac{(1-\eta)(\eta-p+2)}{3-p} \langle \nabla|\nabla u|,\nu\rangle.
\end{split}
\end{equation}
The similar equality also holds on $\Sigma_{t_{2}}$. Together with \eqref{choice-eta} and Proposition~\ref{Regularity-approximation-stability}, we have
\begin{equation}\label{mono-equ-3}
\begin{split}
&\quad \int_{\Sigma_{t_{1}}}\frac{1-\eta}{3-p}t_{1}^{-\b}\Delta u -\int_{\Sigma_{t_{2}}}\frac{1-\eta}{3-p}t_{2}^{-\b}\Delta u \\
&= -\frac{p-2}{3-p}\left( t_{1}^{-\b} F'(t_{1})-t_{2}^{-\b} F'(t_{2})\right)
+\left(t^{-\b}\int_{\Sigma_t} \frac{p-2-\eta}{3-p}\cdot\frac{\eta-1}{\eta}\cdot\Delta u \right)\bigg|_{t_{2}}^{t_{1}} \\[2mm]
&= -(p-2)(G(t_{1})-G(t_{2}))+\mathbf{E}_{t_{1},t_{2}},
\end{split}
\end{equation}
where
\begin{equation}
\mathbf{E}_{t_{1},t_{2}}
= \left(t^{-\b}\int_{\Sigma_t} \frac{p-2-\eta}{3-p}\cdot\frac{\eta-1}{\eta}\cdot\Delta u \right)\bigg|_{t_{2}}^{t_{1}}.
\end{equation}

For the first Ricci term in \eqref{A+B}, we apply co-area formula, the traced Gauss equation and Sard's Theorem to see that
\begin{equation}
\begin{split}
&-\int_{\Omega_{t_1,t_2}} u^{-\b} |\nabla u|^{-1} \Ric(\nabla u,\nabla u)\\
= {} & -\int_{(t_1,t_2)\cap \mathcal{U}} t^{-\b} \left( \int_{\Sigma_t} \Ric(\nu,\nu)\right)\\
= {} & -\int_{(t_1,t_2)\cap \mathcal{U}} t^{-\b} \left( \int_{\Sigma_t} \frac12 R_M -K_{\Sigma_t}-\frac12 |A|^2+\frac12 H^2 \right),
\end{split}
\end{equation}
where $\mathcal{U}$ denotes the set of all regular values of $u$. Using \eqref{regularozed-harmonic-2}, we deduce that for regular value $t\in\mathcal{U}$,
\begin{equation}
\left\{
\begin{array}{ll}
\displaystyle|\nabla u|^2 |A|^2 =|\nabla^2 u|^2 -2|\nabla |\nabla u||^2+ \langle \nabla |\nabla u|, \nu \rangle^2;\\[2mm]
\displaystyle|\nabla u|^2 H^2 =(1+\eta)^2   \langle \nabla |\nabla u|, \nu \rangle^2.
\end{array}
\right.
\end{equation}
Together with the assumptions $R_{M}\geq0$, Lemma~\ref{one-comp} and Guass-Bonnet formula, we deduce that
\begin{equation}\label{mono-equ-4}
\begin{split}
&\quad -\int_{\Omega_{t_1,t_2}} u^{-\b} |\nabla u|^{-1} \Ric(\nabla u,\nabla u) \\
&\leq   \frac{4\pi}{1-\beta}(t_{2}^{-\beta+1}-t_{1}^{-\beta+1})
+\frac{1}{2}\int_{(t_1,t_2)\cap \mathcal{U}} t^{-\beta}\int_{\Sigma_{t}}|\nabla u|^{-2}|\nabla^{2}u|^{2} \\
&\quad  -\int_{(t_1,t_2)\cap \mathcal{U}} t^{-\beta}\int_{\Sigma_{t}}|\nabla u|^{-2}|\nabla|\nabla u||^{2}
-\int_{(t_1,t_2)\cap \mathcal{U}} t^{-\beta}\int_{\Sigma_{t}}\frac{\eta(\eta+2)}{2}|\nabla u|^{-2}\langle\nabla|\nabla u|,\nu\rangle^{2}.
\end{split}
\end{equation}

For the second Ricci term in \eqref{A+B}:
\begin{equation*}
\int_{\Omega_{t_1,t_2}}\frac{\eta-p+2}{3-p} u^{-\b} |\nabla u|^{-1} \Ric(\nabla u,\nabla u)
= \frac{p-2}{3-p}\int_{\Omega_{t_1,t_2}}\frac{\e|\nabla u|}{|\nabla u|^{2}+\e}u^{-\beta}\Ric_{M}(\nu,\nu).
\end{equation*}
Thanks to the Cauchy-Schwarz inequality $2\e^{1/2}|\nabla u|\leq|\nabla u|^{2}+\e$, we obtain
\begin{equation}\label{mono-equ-5}
\int_{\Omega_{t_1,t_2}}\frac{\eta-p+2}{3-p} u^{-\b} |\nabla u|^{-1} \Ric(\nabla u,\nabla u)
\leq  C\e^{1/2}.
\end{equation}

For the terms on the fourth line in \eqref{A+B}, since the critical value of $u$ is of measure zero,
\begin{equation}\label{mono-equ-2}
\begin{split}
&\quad \b\int_{\Omega_{t_1,t_2}}\frac{1-\eta}{3-p} u^{-\b-1} |\nabla u|\left(\left\langle \nabla|\nabla u|,\frac{\nabla u}{|\nabla u|}\right\rangle-\Delta u  \right)\\
&= \b\int_{\Omega_{t_1,t_2}}\frac{1-\eta}{3-p} (1+\eta) u^{-\b-1}  \langle\nabla |\nabla u|,\nabla u \rangle\\
&=\frac{\b(p-1)}{3-p}\int_{\Omega_{t_1,t_2}}(1-\eta) u^{-\b-1}  \langle\nabla |\nabla u|,\nabla u \rangle\\
&\quad + \b\int_{\Omega_{t_1,t_2}} \frac{(1-\eta)(\eta-p+2)}{3-p}u^{-\b-1} \langle\nabla |\nabla u|,\nabla u \rangle\\
&=\frac{\b(p-1)}{3-p}\int_{t_1}^{t_2}t^{-\b-1} F'(t)+ \b\int_{\Omega_{t_1,t_2}} \frac{(1-\eta)(\eta-p+2)}{3-p}u^{-\b-1} \langle\nabla |\nabla u|,\nabla u \rangle\\
&= \frac{\b(p-1)}{3-p} \left(t_2^{-\b-1} F(t_2)-t_1^{-\b-1}F(t_1) \right)+\frac{\b(\b+1)(p-1)}{3-p}\int_{t_1}^{t_2} t^{-\b-2}F(t) \\
&\quad  +\b\int_{\Omega_{t_1,t_2}} \frac{(1-\eta)(\eta-p+2)}{3-p}u^{-\b-1} \langle\nabla |\nabla u|,\nabla u \rangle.
\end{split}
\end{equation}

For the terms in fifth and sixth line in \eqref{A+B}, by \eqref{regularized-harmonic} and co-area formula,
\begin{equation}\label{mono-equ-6}
\begin{split}
& \int_{\Omega_{t_1,t_2}} \left(\frac{-\eta' |\nabla u|}{3-p}-\frac{1-\eta}{3-p}
-\frac{(1-\eta)\eta}{3-p}\right) u^{-\b}\Delta u \cdot |\nabla u|^{-2}\left\langle\nabla |\nabla u|,\nabla u \right\rangle \\
= {} & \int_{(t_1,t_2)\cap \mathcal{U}} t^{-\beta}\int_{\Sigma_{t}}\frac{1-\eta}{3-p}|\nabla u|^{-2}
\left(\frac{\eta'}{1-\eta}|\nabla u|+1+\eta\right)\eta\cdot\langle\nabla|\nabla u|,\nu\rangle^{2}
\end{split}
\end{equation}
and
\begin{equation}\label{mono-equ-7}
\int_{\Omega_{t_1,t_2}} \frac{\eta'}{3-p} u^{-\b} |\nabla |\nabla u||^2
= \int_{(t_1,t_2)\cap \mathcal{U}} t^{-\beta}\int_{\Sigma_{t}}\frac{\eta'}{3-p}\cdot|\nabla u|^{-1}\cdot\big|\nabla|\nabla u|\big|^{2}.
\end{equation}

To simplify our notations, we define
\begin{equation}
\Lambda(t)=\frac{p-1}{3-p} t^{-\b} F'(t) +\frac{\b(p-1)}{3-p}t^{-\b-1}F(t)+\frac{4\pi}{1-\b}t^{-\b+1}.
\end{equation}
Substituting \eqref{mono-equ-3}, \eqref{mono-equ-4}, \eqref{mono-equ-5}, \eqref{mono-equ-2}, \eqref{mono-equ-6} and \eqref{mono-equ-7} into \eqref{A+B},
\begin{equation}\label{computation 1}
\begin{split}
& \Lambda(t_1)-\Lambda(t_2) \\
\leq {} & -\int_{(t_1,t_2)\cap \mathcal{U}}t^{-\beta}\int_{\Sigma_{t}}\left(\frac{1-\eta}{3-p}-\frac{1}{2}\right)\cdot|\nabla u|^{-2}\cdot
|\nabla^{2}u|^{2} \\
& +\int_{(t_1,t_2)\cap \mathcal{U}} t^{-\beta}\int_{\Sigma_{t}}\frac{p-2-\eta}{3-p}\cdot |\nabla u|^{-2}\cdot|\nabla|\nabla u||^{2}\\
& -\int_{(t_1,t_2)\cap \mathcal{U}} t^{-\beta}\int_{\Sigma_{t}}\frac{\eta(\eta+2)}{2}\cdot|\nabla u|^{-2}\cdot\langle\nabla|\nabla u|,\nu\rangle^{2}\\
& +\int_{(t_1,t_2)\cap \mathcal{U}} t^{-\beta}\int_{\Sigma_{t}}\frac{1-\eta}{3-p}\cdot|\nabla u|^{-2}\cdot
\left(\frac{\eta'}{1-\eta}\cdot|\nabla u|+1+\eta\right)\cdot\eta\cdot\langle\nabla|\nabla u|,\nu\rangle^{2} \\
& +\int_{(t_1,t_2)\cap \mathcal{U}} t^{-\beta}\int_{\Sigma_{t}}\frac{\eta'}{3-p}\cdot|\nabla u|^{-1}\cdot\big|\nabla|\nabla u|\big|^{2} \\
& +\b \int_{(t_1,t_2)\cap \mathcal{U}} t^{-\b-1} \int_{\Sigma_t}\frac{(1-\eta)(\eta+2-p)}{3-p}  \langle \nabla|\nabla u|,\nu\rangle\\
& +\frac{\beta(\beta+1)(p-1)}{3-p}\int_{(t_1,t_2)\cap \mathcal{U}}t^{-\beta-2}F(t)+\mathbf{E}_{t_1,t_2}+C\sqrt{\e} \\
= {} & \sum_{k=1}^6\mathbf{I}_k+\frac{\beta(\beta+1)(p-1)}{3-p}\int_{(t_1,t_2)\cap \mathcal{U}} t^{-\beta-2}F(t)+\mathbf{E}_{t_1,t_2}+C\sqrt{\e}.
\end{split}
\end{equation}

\medskip
\noindent
{\bf Step 3.} Estimate each term in \eqref{computation 1}.
\medskip

Using $1<p\leq2$ and $p-2\leq\eta\leq0$, we obtain
\begin{equation}
\frac{1-\eta}{3-p}-\frac{1}{2} \geq 0, \ \
\frac{\eta-p+2}{3-p} \geq 0.
\end{equation}
By the improved Kato inequality (Lemma \ref{improved-Kato}) and $\langle \nabla|\nabla u|,\nu\rangle^2\leq|\nabla^{2}u|^{2}$,
\begin{equation}\label{Kato}
\mathbf{I}_1\leq -\int_{(t_1,t_2)\cap \mathcal{U}} t^{-\b} \int_{\Sigma_t} \left( \frac{1-\eta}{3-p}-\frac12 \right) \frac{\eta^2+2\eta+3}{2} |\nabla u|^{-2} \langle \nabla|\nabla u|,\nu\rangle^2
\end{equation}
and
\begin{equation}
\mathbf{I}_{2} \leq -\int_{(t_1,t_2)\cap \mathcal{U}} t^{-\b} \int_{\Sigma_t} \frac{\eta-p+2}{3-p} |\nabla u|^{-2} \langle \nabla|\nabla u|,\nu\rangle^2.
\end{equation}

For the terms involving $\eta'$, since $\eta'\leq 0$ and $|\langle \nabla |\nabla u|,\nu \rangle |\leq |\nabla |\nabla u||$, we have 
\begin{equation}
    \frac{\eta'}{3-p}|\nabla u|^{-1} |\nabla |\nabla u||^2+ \frac{\eta' \eta}{3-p}|\nabla u|^{-1} |\langle \nabla |\nabla u|,\nu \rangle^2\leq 0.
\end{equation}
and hence 
\begin{equation}
\sum_{k=1}^5\mathbf{I}_k 
\leq \int_{(t_1,t_2)\cap \mathcal{U}} t^{-\b} \int_{\Sigma_t} \Upsilon\cdot |\nabla u|^{-2} \langle \nabla|\nabla u|,\nu\rangle^2
\end{equation}
where
\begin{equation*}
\begin{split}
\Upsilon&=-\left( \frac{1-\eta}{3-p}-\frac12 \right)\cdot \frac{\eta^2+2\eta+3}{2}-\frac{\eta(\eta+2)}{2}
+\frac{p-2-\eta}{3-p}+\frac{\eta(1-\eta)}{3-p}\left(1+\eta \right)\\
&=-\frac{1}{2}\cdot\frac{1-\eta}{3-p}\cdot\frac{\eta^2+2\eta+3}{2}
-\frac{1-\eta}{3-p}\cdot\frac{\eta^{2}+2\eta}{2}
-\frac{1}{4}\cdot\frac{\eta+2-p}{3-p}(1+\eta)^{2}
+\frac{(1-\eta)}{3-p}\left(\eta+\eta^{2} \right)\\
&=-\frac{1-\eta}{4(3-p)}\cdot (3-\eta)(\eta+1)-\frac{\eta+2-p}{4(3-p)}(1+\eta)^{2}.
\end{split}
\end{equation*}
By $p-2\leq\eta\leq0$, we see that
\begin{equation}
\begin{split}
\Upsilon \leq {} & -\frac{1-\eta}{4(3-p)}(5-p)(p-1)- \frac{\eta+2-p}{4(3-p)}(1+\eta)^{2} \\
\leq {} & -\frac{(5-p)(p-1)}{4(3-p)^2}(1-\eta)^2-\frac{\eta+2-p}{4(3-p)}(1+\eta)^{2}.
\end{split}
\end{equation}
Then
\begin{equation}\label{I 1-5}
\begin{split}
\sum_{k=1}^5\mathbf{I}_k
\leq {} & -\frac{(5-p)(p-1)}{4(3-p)^2}\cdot\int_{(t_1,t_2)\cap \mathcal{U}}t^{-\b} \int_{\Sigma_s} (1-\eta)^2|\nabla u|^{-2} \langle \nabla|\nabla u|,\nu \rangle^2 \\
&-\int_{(t_1,t_2)\cap \mathcal{U}} t^{-\b} \int_{\Sigma_t} \frac{\eta-p+2}{4(3-p)}(1+\eta)^{2} |\nabla u|^{-2} \langle \nabla|\nabla u|,\nu \rangle^2
\end{split}
\end{equation}
On the other hand,
\begin{equation}\label{3rd-ORDER}
\begin{split}
\mathbf{I}_6 &\leq \int_{(t_1,t_2)\cap \mathcal{U}} t^{-\b} \int_{\Sigma_t} \frac{\eta-p+2}{4(3-p)}(1+\eta)^{2} |\nabla u|^{-2} \langle \nabla|\nabla u|,\nu \rangle^2\\
&\quad +\beta^{2}\int_{\Omega_{t_1,t_2}} u^{-\b-2}  \frac{(1-\eta)^2(\eta-p+2)}{(1+\eta)^{2}(3-p)} |\nabla u|^3.
\end{split}
\end{equation}
Substituting \eqref{I 1-5} and \eqref{3rd-ORDER} into \eqref{computation 1},
\begin{equation*}
\begin{split}
\Lambda(t_{1})-\Lambda(t_{2})&\leq -\frac{(5-p)(p-1)}{4(3-p)^2}\int_{(t_1,t_2)\cap \mathcal{U}}
t^{-\b} \int_{\Sigma_t} (1-\eta)^2|\nabla u|^{-2} \langle \nabla|\nabla u|,\nu \rangle^2\\
&\quad +\beta^{2}\int_{\Omega_{t_1,t_2}} u^{-\b-2}  \frac{(1-\eta)^2(\eta-p+2)}{(1+\eta)^{2}(3-p)} |\nabla u|^3\\
&\quad +\frac{\beta(\beta+1)(p-1)}{3-p}\int_{(t_1,t_2)\cap \mathcal{U}} t^{-\beta-2}F(t)+\mathbf{E}_{t_1,t_2}+C\e^{1/2}.
\end{split}
\end{equation*}
The second term can be controlled by $C\e$ thanks to Proposition \ref{Regularity-approximation-stability}. For the first term, the Cauchy-Schwarz inequality
\begin{equation}\label{CS}
\begin{split}
\left|F'(t)\right| = {} & \left|\int_{\Sigma_t}(1-\eta)\langle \nabla|\nabla u|,\nu\rangle\right| \\
\leq {} & F(t)^{1/2} \cdot \left(\int_{\Sigma_t} (1-\eta)^{2}|\nabla u|^{-2}\langle \nabla|\nabla u|,\nu\rangle^2 \right)^{1/2}
\end{split}
\end{equation}
shows
\begin{equation}
-(1-\eta)^2|\nabla u|^{-2} \langle \nabla|\nabla u|,\nu \rangle^2
\leq -F^{-1}(t)\cdot(F'(t))^{2}.
\end{equation}
Then
\begin{equation}\label{errorterm2}
\begin{split}
\Lambda(t_1)-\Lambda(t_2)
&\leq -\frac{(5-p)(p-1)}{4(3-p)^{2}}\int_{(t_1,t_2)\cap \mathcal{U}}t^{-\beta}\cdot F^{-1}(t)\cdot(F'(t))^{2} \\
& \quad +\frac{\beta(\beta+1)(p-1)}{3-p}\int_{t_1}^{t_2}t^{-\beta-2}F(s)+\mathbf{E}_{t_1,t_2}+C\e^{1/2}.
\end{split}
\end{equation}
Using the Cauchy-Schwarz inequality
\begin{equation}\label{square}
2t^{-1}\lambda F'(t) \leq  F^{-1}(t)(F'(t))^2+ \lambda^2 t^{-2}F(t)
\end{equation}
where $\lambda$ satisfies
\begin{equation}\label{Choice-lambda}
0=-2\lambda(5-p)(\b+1)+4(3-p)\b(\b+1)+\lambda^2(5-p),
\end{equation}
we obtain
\begin{equation}
\begin{split}
& \quad \Lambda(t_1)-\Lambda(t_2)\\
&\leq  -\frac{(5-p)(p-1)}{4(3-p)^{2}}\int_{(t_1,t_2)\cap \mathcal{U}}\left(2t^{-\beta-1}\lambda F'(t)-\lambda^{2}t^{-2-\beta}F(t)\right) \\
&\quad  +\frac{\beta(\beta+1)(p-1)}{3-p}\int_{(t_1,t_2)\cap \mathcal{U}}t^{-\beta-2}F(t)+\mathbf{E}_{t_1,t_2}+C\e^{1/2}\\
&=\frac{\lambda (p-1)(5-p)}{2(3-p)^2} \left(t_{1}^{-\b-1}F(t_{1})-t_{2}^{-\b-1}F(t_{2}) \right)+\mathbf{E}_{t_1,t_2}+C\e^{1/2}.
\end{split}
\end{equation}
This completes the proof by rearranging.
\end{proof}

\subsection{Some discussions}
Let $t_{1}<t_{2}$ be two regular values of the original $p$-Green function $\hat{u}$. If $p\in(1,3)$ and $\hat{u}$ is smooth on
\begin{equation}
\Omega_{t_1,t_2}=\{x\in M: t_1\leq \hat{u}(x)\leq t_2\},
\end{equation}
then the same argument of Theorem \ref{monotonicity-regularized-THM} (without additional regularization) shows that the function
\begin{equation}
\mathcal{H}(t)= t^{-\b} F'(t)+\left(\b-\frac{\lambda(5-p)}{2(3-p)}\right)t^{-\b-1}F(t) -\frac{4\pi(3-p)}{(p-1)(\b-1)}t^{-\b+1}
\end{equation}
satisfies
\begin{equation}
\mathcal{H}(t_{1}) \leq \mathcal{H}(t_{2}).
\end{equation}
For later use, let us point out the case when the equality holds. When $\mathcal{H}(t_{1})=\mathcal{H}(t_{2})$, the equalities of \eqref{mono-equ-4}, \eqref{Kato}, \eqref{CS} and \eqref{square} hold. Then for any regular values $t\in[t_1,t_2]$, we have the following:
\begin{enumerate}\setlength{\itemsep}{1mm}
\item[(i)]  $R_M= 0$ on $\S_t$ and $\int_{\S_t}{K_{\S_t}} =  4\pi$. Hence, $\S_t$ is topologically a sphere;
\item[(ii)] For each point on $\S_t$, if we choose $\{e_j\}_{j=1}^3$ such that $e_1=\nu=\frac{\nabla\hat{u}}{|\nabla\hat{u}|}$, then we have
\begin{equation*}\label{traceless2ndff}
\begin{split}
\hat{u}_{12}= {} & \hat{u}_{13}=\hat{u}_{23}=0,\\
\hat{u}_{22}= {} & \hat{u}_{33}=-\frac{p-1}{2}\hat{u}_{11}=\frac{1}{2}H|\Na \hat{u}|.
\end{split}
\end{equation*}
It follows that $\Na^{\S}|\Na \hat u|=0$ on $\S_t$ and the second fundamental form of $\S_t$ has vanishing traceless part;
\item[(iii)] There exists $A_t\in \R$ such that on $\S_t$,
\begin{equation*}\label{Hu}
|\Na \hat u|^{2} = A_t \la\Na |\Na \hat u|,\nu \ra;
\end{equation*}
\item[(iv)] $F'(t)=\lambda t^{-1}F(t)$.
\end{enumerate}

\section{Proof of Main Theorems}

The monotonicity in Theorem~\ref{main-Thm} and Theorem \ref{generalized F Thm} will follow from the following slightly more general theorem. The flexibility will allow us to study the case of harmonic function as proposed in \cite{ChodoshLi2021}.

\begin{thm}\label{generalized-main}
Under the assumptions of Theorem \ref{main-Thm}, if $\lambda>0$ and $\b>1$ satisfy  $(5-p)+(3p-7)\b\ge 0$ and
\begin{equation}
\frac{\lambda(5-p)(\b+1)}{2(3-p)}= \b(\b+1)+\frac{\lambda^2 (5-p)}{4(3-p)},
\end{equation}
then we have
\begin{equation}\label{mono-1-new}
    \mathcal{G}(t)\leq \frac{3-p}{p-1}\cdot\frac{4\pi}{\b-1} t
\end{equation}
for all $t>0$ which is a regular value of $\hat u$, 
where $$\mathcal{G}(t)= F'(t)+  \left[\b -\frac{\lambda (5-p)}{2(3-p)}\right]t^{-1} F(t). $$
Moreover,
\begin{equation}\label{mono-2-new}
\mathcal{I}(t) = t^{\b -\frac{\lambda (5-p)}{2(3-p)}}F(t)
-\frac{4\pi}{(\b-1)(\b -\frac{\lambda (5-p)}{2(3-p)}+2)}\left(\frac{3-p}{p-1}\right)t^{\b -\frac{\lambda (5-p)}{2(3-p)}+2}
\end{equation}
is non-increasing for all $t>0$.
\end{thm}

\begin{proof}
For \eqref{mono-1-new}, let $t<\tau$ be two regular values of $\hat{u}$. We first show that $\mathcal{H}(t)\leq \mathcal{H}(\tau)$, where
\begin{equation}
\mathcal{H}(s)=s^{-\b} F'(s)+\left(\b-\frac{\lambda(5-p)}{2(3-p)}\right)s^{-\b-1}F(s) -\frac{4\pi(3-p)}{(p-1)(\b-1)}s^{-\b+1}.
\end{equation}
By Proposition \ref{Regularity-approximation-stability}, when $\e$ is sufficiently small, $t$ and $\tau$ are also regular values of $u_{\e}$. Moreover, $u_{\e}\to\hat{u}$ in $C^{\infty}(U)$ for some neighborhood $U$ of $\Sigma_{t}\cup\Sigma_{\tau}$. This implies
\begin{equation}
\lim_{\e\to0}F_{\e}(t) = F(t), \ \
\lim_{\e\to0}F_{\e}(\tau) = F(\tau),
\end{equation}
\begin{equation}
\lim_{\e\to0}F_{\e}'(t) = F'(t), \ \
\lim_{\e\to0}F_{\e}'(\tau) = F'(\tau), \ \
\lim_{\e\to0}\mathbf{E}_{\e,t,\tau} = 0
\end{equation}
Combining the above with Theorem \ref{monotonicity-regularized-THM}, we obtain $\mathcal{H}(t)\leq \mathcal{H}(\tau)$. By Proposition \ref{asymptotic-Green-pole}, a sufficiently large number must be a regular value of $\hat{u}$, and then
\begin{equation}
\mathcal{H}(t) \leq \lim_{\tau\to+\infty}\mathcal{H}(\tau)=0,
\end{equation}
which implies \eqref{mono-1-new}.

For the monotonicity of $\mathcal{I}(t)$, since $F$ is not known to be absolutely continuous, we consider $F_{\e}$ instead. We first show the almost monotonicity of $\mathcal{I}_{\e}(t)$, where
\begin{equation}
\mathcal{I}_{\e}(t) = t^{\b -\frac{\lambda (5-p)}{2(3-p)}}F_{\e}(t)
-\frac{4\pi}{(\b-1)(\b -\frac{\lambda (5-p)}{2(3-p)}+2)}\left(\frac{3-p}{p-1}\right)t^{\b -\frac{\lambda (5-p)}{2(3-p)}+2}.
\end{equation}
As in the discussion of Section~\ref{sec: pre}, we will work on $D$, which contains
\begin{equation}
\{a<\hat{u}<b\} \Subset D.
\end{equation}
Choosing $b$ sufficiently large such that $\tau>b-2$ are all regular values of $\hat{u}$. When $\e$ is sufficiently small, $\tau>b-1$ are all regular values of $u_{\e}$.
Fix $\e$ and let $t_{1}<t_{2}$ be two regular values of $u_{\e}$ such that $t_{1},t_{2}\in(a,b-3)$. Thanks to Lemma \ref{local Lipschitz}, $F_{\e}$ is locally Lipschitz. Write $\a=\b-\frac{\lambda (5-p)}{2(3-p)}$. We can apply fundamental theorem of calculus to obtain
\begin{equation}\label{I e monotonicity computation}
\begin{split}
t^\a F_\e(t)\Big|_{t_1}^{t_2} &=\int_{(t_1,t_2)\cap \mathcal{U}}\left(t^\a F_\e(t)\right)' dt\\
&\leq \int_{(t_1,t_2)\cap \mathcal{U}} t^{\b+\a}\left(\frac{3-p}{p-1} \cdot\frac{4\pi }{\b-1}  t^{-\beta+1}
+\mathcal{H}_\e(\tau)+\mathbf{E}_{\e,t,\tau}+C\e^{1/2}\right) dt,
\end{split}
\end{equation}
where $\mathcal{U}$ denotes the set of regular value of $u_\e$. By co-area formula,
\begin{equation}
\begin{split}
\int_{(t_1,t_2)\cap \mathcal{U}} t^{\b+\a}\mathbf{E}_{\e,t,\tau}\,dt
= {} &\int_{(t_1,t_2)\cap \mathcal{U}} t^{\a}\left(\int_{\Sigma_t} \frac{(p-2-\eta_{\e})(\eta_{\e}-1)}{(3-p)\eta_{\e}} \Delta u_\e \right) dt\\
= {} & \int_{\Omega_{\e,t_1,t_2}}  \frac{(p-2-\eta_{\e})(1-\eta_{\e})}{(3-p)}u_\e^{\a} \langle\nabla |\nabla u_\e|,\nabla u_\e \rangle \\
\leq {} & C\int_{\Omega_{\e,t_1,t_2}}  |\nabla u_\e|\cdot|p-2-\eta_{\e}|\cdot|\nabla^{2}u_{\e}|,
\end{split}
\end{equation}
for some $C$ independent of $\e$. Using \eqref{choice-eta}, we see that
\begin{equation}
|\nabla u_\e|\cdot|p-2-\eta_{\e}|=\frac{\e|p-2|\cdot|\nabla u_\e|}{  |\nabla u_\e|^2+\e}\leq \frac{1}{2}\e^{1/2}|p-2|.
\end{equation}
Combining this with Proposition \ref{Regularity-approximation-stability} (ii),
\begin{equation}
\int_{(t_1,t_2)\cap \mathcal{U}} t^{\b+\a}\mathbf{E}_{\e,t,\tau}\,dt \leq C\e^{1/2}.
\end{equation}
Substituting this into \eqref{I e monotonicity computation},
\begin{equation}\label{I e monotonicity}
\mathcal{I}_{\e}(t_{2}) \leq \mathcal{I}_{\e}(t_{1})+C(t_{2}-t_{1})|\mathcal{H}_{\e}(\tau)|+C\e^{1/2}.
\end{equation}
Thanks to Sard's theorem and Lemma \ref{local Lipschitz}, \eqref{I e monotonicity} holds for all $t_{1},t_{2}\in(a,b-3)$. Recalling that $\tau$ is regular of $\hat{u}$, we have $\lim_{\e\to0}\mathcal{H}_{\e}(\tau)=\mathcal{H}(\tau)$. Letting $\e\to0$ in \eqref{I e monotonicity} and using Lemma \ref{convergence F e and F}, we obtain
\begin{equation}
\mathcal{I}(t_{2}) \leq \mathcal{I}(t_{1})+C(t_{2}-t_{1})|\mathcal{H}(\tau)|.
\end{equation}
Since $a$ and $b$ are arbitrary, letting $\tau\to+\infty$ and using $\lim_{\tau\to+\infty}\mathcal{H}(\tau)=0$,
\begin{equation}
\mathcal{I}(t_{2}) \leq \mathcal{I}(t_{1}).
\end{equation}
We obtain the monotonicity of $\mathcal{I}$.
\end{proof}

\begin{rem}\label{p 2 3}
When $p\in(2,3)$ and $\hat{u}$ is smooth, Theorem \ref{main-Thm} can be proved by the discussion in Section 3.3 and the above argument. Unlike the general case, we cannot apply regularization method to prove Theorem \ref{main-Thm} unless the extra regularity is assumed. This is because the choice of regularization (i.e. choice of $\phi_\e$) fails to be controlled (especially Step 3) when $p>2$.
\end{rem}

Now the monotonicity in Theorem~\ref{main-Thm} and Theorem \ref{generalized F Thm} follow directly from Theorem~\ref{generalized-main} by taking $(\lambda,\b)=(2,\frac{2}{3-p})$. Indeed this is the Lagrangian multiplier solution which minimizes the decay rate. In particular, this decay rate is corresponding to the Euclidean model. Next we would like to study the rigidity case in the monotonicity.

\begin{proof}[Proof of rigidity in Theorem \ref{main-Thm}]
We first consider the rigidity {\bf (a')}. Equivalently, we may rewrite the assumption as $\mathcal{H}(t_0)=0$. By Proposition~\ref{asymptotic-Green-pole}, all $\tau\gg1$ are regular and $\mathcal{H}(\tau)\to 0$ as $\tau\to +\infty$. Theorem~\ref{main-Thm} implies that for all regular value $t>t_0$, we must have $\mathcal{H}(t)=0$. Since $u\in C^{1,\alpha}_{\loc}$, there exists a maximal open interval $(a,b)$ containing $t_0$ such that $t$ is a regular value of $\hat{u}$ for all $t\in (a,b)$.
Let $t\in I:=[t_0,b)$. Since $t$ is regular, we may apply the (i)-(iv) in Section 3.3.

Write $\nu=\frac{\nabla\hat{u}}{|\nabla \hat{u}|}$. Recall that on the regular level set $\Sigma_{t}$, we have
\begin{equation}\label{rigidity eqn 1}
H|\Na \hat{u}|=-(p-1)\la \Na |\Na \hat{u}|, \nu\ra
\end{equation}
and so
\begin{equation}
F'(t) = -\frac{3-p}{p-1}\int_{\Sigma_{t}}H|\nabla\hat{u}|
= (3-p)\int_{\Sigma_{t}}\la \Na |\Na \hat{u}|, \nu\ra.
\end{equation}
Combining this with (iii) and (iv),
\begin{equation}
2 t^{-1}F(t) = F'(t)
= (3-p)A_{t}^{-1}\int_{\Sigma_{t}}|\nabla\hat{u}|^{2}
= (3-p)A_{t}^{-1}F(t).
\end{equation}
Since $F(t)\neq0$, then
\begin{equation}
A_{t} = \frac{3-p}{2}t.
\end{equation}
Using \eqref{rigidity eqn 1} and (iii),
\begin{equation}
H|\Na \hat{u}|=-(p-1)A_{t}^{-1}|\nabla\hat{u}|^{2}.
\end{equation}
Hence, on $\Sigma_{t}$,
\begin{equation}\label{gradupreODE}
\frac{|\nabla\hat{u}|}{\hat{u}} = \frac{p-3}{2(p-1)}H.
\end{equation}
Thanks to (i), $\Na^{\S}|\Na\hat{u}|=0$ on $\S_t$, it follows that $|\Na\hat{u}|=f(\hat{u})$ for some $f:I\to(0,\infty)$ and hence $H$ is also a function of $\hat{u}$. In particular, we have
\begin{equation}\label {HODE}
H=-(p-1)\frac{\p |\Na\hat{u}|}{\p \hat u}=-(p-1)f'(\hat{u}).
\end{equation}
As a result, \eqref{gradupreODE} can be expressed as the following ODE:
\begin{equation}\label{rigidity eqn 2}
\frac{f(\hat{u})}{\hat{u}}=\frac{3-p}{2}f'(\hat{u}).
\end{equation}
Hence, there exists $C>0$ such that
\begin{equation}\label{gradufu}
|\Na \hat{u}|=f(\hat{u})=C\hat{u}^{\frac{2}{3-p}}.
\end{equation}
By (i), we know that for all $t\in I$, $\S_t$ has one component only, $P_b:=\{\hat{u}\geq b\}$ is bounded by $\widetilde\S_b:=\{x\in M: x=\lim_{i\to+\infty}x_i, \ \text{where } \hat{u}(x_i)\in I \}$. Assume that $b<\infty$, by continuity of $|\Na \hat{u}|$ and \eqref{gradufu}, there exists $y\in P_b\setminus \widetilde\S_b$ such that $\hat{u}(y)=b$ and $\Na \hat{u}(y)=0$ since $b$ is maximal. Since $\de P_b \subset \widetilde{\Sigma}_{b}$, then $\hat{u}$ achieves an interior minimum in $P_b$. By the Harnack inequality in \cite{WangZhang2011} and hence strong maximum principle, $\hat{u}\equiv b$ on $P_b\setminus \{x_0\}$. Contradiction arises. We can therefore conclude that $I=[t_0,\infty)$. Hence, the superlevel set $\{\hat{u}>t _0\}$ is topologically a ball foliated by level set spheres.

On the other hand, \eqref{rigidity eqn 2} and \eqref{gradufu} shows
\begin{equation}
H = -\frac{2C(p-1)}{3-p}\hat{u}^{\frac{p-1}{3-p}}.
\end{equation}
Combining this with (ii), we see that the function
\begin{equation}
Q=\frac1{2 C^2} \left(\frac{3-p}{p-1}\right)^2\hat{u}^{-\frac{2(p-1)}{3-p}}
\end{equation}
satisfies $\nabla^2 Q=g$ on $\{u>t _0\}$ and hence Ricci identity implies for all $i,j,k$,
\begin{equation}\label{vanishing curvature}
R_{ijk}\,^l Q_l = Q_{;kij}-Q_{;kji}=0.
\end{equation}
We choose $\{e_i\}_{i=1}^3$ so that $e_1=\frac{\nabla\hat{u}}{|\nabla\hat{u}|}$. It is clear that $\nabla Q$ is parallel to $e_{1}$. Then \eqref{vanishing curvature} implies $R_{1221}=R_{1331}=0$. By (i), we know $R_M=0$ and so $R_{2332}=0$. Since $\{e_2,e_3\}$ is arbitrary, $\mathrm{Rm}\equiv 0$ on $\{\hat{u}>t_0\}$. Combining this with the foliation, we obtain the isometry.

\medskip

We now consider the rigidity {\bf (b')}. For any regular value $t$ of $\hat{u}$, define
\begin{equation}
\mF(t):=t^{-1}F(t)- 4\pi\left(\frac{3-p}{p-1}\right)^2 t.
\end{equation}
Then the assumption is equivalent to say that
\begin{equation}
   \mF(t_0)=\mF(s_0)
\end{equation}
for some regular values $s_0<t_0$. The monotonicity {\bf (a)} implies that $\mF(t)$ is constant for all regular values $s_0<t<t_0$. Since $u\in C^{1,\alpha}_{\loc}$, there exists a maximal open interval $(a,b)$ containing $s_0$ such that $t$ is a regular value for $\hat{u}$ for all $t\in (a,b)$. We may differentiate to obtain $\mathcal{F}'(s_0)=0$ which in turn implies  $\mH(s_0)=0$ and hence the result follows from the rigidity {\bf (a')}.  
\end{proof}

We now establish the upper bound of $F(t)$ under rough Ricci lower bound. Basically, Corollary~\ref{main-monot-zerothorder} follows from Theorem \ref{generalized F Thm} and gradient estimate \cite[Theorem 1.1]{WangZhang2011}.

\begin{proof}[Proof of Corollary~\ref{main-monot-zerothorder}]
For monotonicity {\bf (c)}, we first consider $p\in (1,2)$ which does not require perturbation argument. Recall from \eqref{bound gradient p-1} that for all $t\in(0,1)$,
\begin{equation}
\int_{\{\hat{u}=t\}\cap\{|\nabla \hat{u}|>0\}}|\nabla\hat{u}|^{p-1} \leq C
\end{equation}
for some constant $C$ independent of $\e$. Together with the gradient estimate \cite[Theorem 1.1]{WangZhang2011}, we deduce that for any $t\in (0,1)$,
\begin{equation}\label{F upper bound}
    \begin{split}
        F(t)&=\lim_{\delta\to 0}\int_{\Sigma_t\cap \{|\nabla \hat{u}|>\delta\}}|\nabla\hat{u}|^2\\
        &=\lim_{\delta\to 0}\int_{\Sigma_t\cap \{|\nabla \hat{u}|>\delta\}}|\nabla \hat{u}|^{p-1}\cdot |\nabla \hat{u}|^{3-p}\\[1.5mm]
        &\leq C\sup_{\Sigma_t} |\nabla \hat{u}|^{3-p}\\[1.5mm]
        &\leq C t^{3-p}.
    \end{split}
\end{equation}
If $1<p<2$, it is clear that $t^{-1}F(t)\to0$ as $t\to0$. Combining this with Theorem \ref{generalized F Thm}, we obtain {\bf (c)}.

If $p=2$, we use the observation in \cite[Section 4]{ChodoshLi2021}. Choose $\b_i$ and $\lambda_i$ so that
\begin{equation}
\frac{3\lambda_i(\b_i+1)}{2}=\b_i(\b_i+1)+\frac{3\lambda_i^2}{4},
\end{equation}
$\b_i\to 2, \lambda_i\to 2$ and $\b_i-\frac{3\lambda_i}{2}$ decrease to $-1$. Since the assumptions on $\b,\lambda$ are open conditions, this is always do-able, see \cite[Lemma 15]{ChodoshLi2021}. By the same argument of \eqref{F upper bound},
\begin{equation}
    t^{\b_i-\frac{3\lambda_i}{2}}F(t) \to 0, \ \ \text{as $t\to 0$}
\end{equation}
for all $i>0$. Then {\bf (c)} follows from the monotonicity of $t^{\b_i-\frac{3\lambda_i}{2}}F(t)$ by letting $t\to 0$ and followed by $i\to +\infty$.

To prove monotonicity {\bf (d)}, we adapt the argument in \cite{MunteanuWang2021}. Thanks to Lemma \ref{gradient p-1 1} and monotonicity {\bf (c)},
\begin{equation}\label{monotonicity d}
    \begin{split}
    1&=\int_{\Sigma_t}|\nabla\hat{u}|^{p-1}\\
    &\leq \left(\int_{\Sigma_t}|\nabla\hat{u}|^2\right)^\frac{p-1}{2} \mathrm{Area}(\Sigma_t)^\frac{3-p}{2}\\
    &\leq \left[4\pi \left(\frac{3-p}{p-1} \right)^2t^2\right]^\frac{p-1}{2}\mathrm{Area}(\Sigma_t)^\frac{3-p}{2},
    \end{split}
\end{equation}
which implies {\bf (d)}.

Thanks to \eqref{monotonicity d}, at a regular value $t_{0}$, it is clear that {\bf(c)} holds if and only if {\bf (d)} holds. Then it suffices to prove the rigidity of {\bf (c)}. Let $I=(a,b)$ be the maximal connected open interval containing $t_0$ so that $t$ is regular for all $t\in I$. Define
\[
\mathcal{F}(t) = t^{-1}F(t)-4\pi\left(\frac{3-p}{p-1}\right)^2t.
\]
Then {\bf (c)} holds at $t_{0}$ implies $\mathcal{F}(t_{0})=0$. Thanks to Theorem \ref{generalized F Thm}, we obtain $\mathcal{F}(t)=0$ for all $t\in(0,t_{0}]$ and so $\mathcal{F}'(t_{0})=0$. Then Theorem \ref{main-Thm} {\bf (a')} implies that $b=+\infty$. By the slight modification of the proof of Theorem \ref{main-Thm} {\bf (a')}, we can show that $a=0$. We sketch the argument here. It is clear that $\mathcal{F}'(t)=0$ for all $t\in (a,t_{0})$. It then follows that $\mathcal{H}(t)=0$ for all $t\in (a,t_{0})$. The same argument of Theorem \ref{main-Thm} {\bf (a')} shows \eqref{gradufu} holds up to the closure of $I$. By the maximality of the interval $I$ and maximum principle, $a$ must be zero. Using Theorem \ref{main-Thm} {\bf (a')} again, $(M,g)$ is isometric to the Euclidean space.
\end{proof}


\begin{thebibliography}{10}

\bibitem{AFM20} Agostiniani, V.; Fogagnolo, M.; Mazzieri, L., {\sl Sharp geometric inequalities for closed hypersurfaces in manifolds with nonnegative Ricci curvature}, Invent. Math. 222 (2020), no. 3, 1033--1101.

\bibitem{AFM19} Agostiniani, V.; Fogagnolo, M.; Mazzieri, L., {\sl Minkowski Inequalities via Nonlinear Potential Theory}, preprint, arXiv:1906.00322.

\bibitem{AM17} Agostiniani, V.; Mazzieri, L., {\sl On the geometry of the level sets of bounded static potentials}, Commun. Math. Phys 355 (2017), 261--301.

\bibitem{AM20} Agostiniani, V.; Mazzieri, L., {\sl Monotonicity formulas in potential theory}, Calc. Var. Partial Differential Equations 59 (2020), no. 1, Paper No. 6, 32 pp.

\bibitem{AMO2021} Agostiniani, V.; Mazzieri, L.; Oronzio, F, {\sl A Green's function proof of the Positive Mass Theorem}, preprint, arXiv:2108.08402.

\bibitem{BKKS2019} Bray, H.; Kazaras, D.; Khuri, M; Stern, D., {\sl Harmonic functions and the mass of 3-dimensional asymptocially flat Riemannian manifolds}, arXiv:1911.06754.

\bibitem{CNV2015} Cheeger, J.; Naber, A.; Valtorta, D., {\sl Critical sets of elliptic equations}, Comm. Pure Appl. Math. 68 (2015), no. 2, 173--209.

\bibitem{Cheng1975} Cheng, S. Y., {\sl Eigenvalue comparison theorems and its geometric applications}, Math. Z. 143 (1975), no. 3, 289--297.

\bibitem{ChengYau1975} Cheng, S. Y.; Yau, S. T., {\sl  Differential equations on Riemannian manifolds and their geometric applications}, Comm. Pure Appl. Math. 28 (1975), no. 3, 333--354.

\bibitem{ChodoshLi2021} Chodosh, O.; Li, C., {\sl Stable minimal hypersurfaces in $\mathbb{R}^4$}, preprint, arXiv:2108.11462.

\bibitem{Colding2012} Colding, T. H., {\sl  New monotonicity formulas for Ricci curvature and applications. I}, Acta Math. 209 (2012), no. 2, 229--263.

\bibitem{ColdingMinicozzi2013} Colding, T. H.; Minicozzi, W. P., II, {\sl  Monotonicity and its analytic and geometric implications}, Proc. Natl. Acad. Sci. USA 110 (2013), no. 48, 19233--19236.

\bibitem{ColdingMinicozzi2014} Colding, T. H.; Minicozzi, W. P., II, {\sl  Ricci curvature and monotonicity for harmonic functions}, Calc. Var. Partial Differential Equations 49 (2014), no. 3-4, 1045--1059.

\bibitem{ColdingMinicozzi2014b} Colding, T. H.; Minicozzi, W. P., II, {\sl  On uniqueness of tangent cones for Einstein manifolds}, Invent. Math. 196 (2014), no. 3, 515--588.

\bibitem{GarofaloLin1986} Garofalo, N.; Lin, F.-H., {\sl  Monotonicity properties of variational integrals, $A_p$ weights and unique continuation}. Indiana Univ. Math. J. 35 (1986), no. 2, 245--268.


\bibitem{GarofaloLin1987} Garofalo, N.; Lin, F.-H., {\sl Unique continuation for elliptic operators: a geometric-variational approach}, Comm. Pure Appl. Math. 40 (1987), no. 3, 347--366.



\bibitem{Grigoryan1983}
Grigor'yan, A., {\sl On the existence of a Green function on a manifold} (in Russian), Uspekhi Mat. Nauk 38 (1983), 161--162; English transl. in Russian Math. Surveys 38 (1983),
190--191.

\bibitem{Grigoryan1985}
Grigor'yan, A.,  {\sl On the existence of positive fundamental solutions of the Laplace equation on
Riemannian manifolds} (in Russian), Mat. Sb. (N.S.) 128 (1985), 354--363; English
transl. in Math. USSR-Sb. 56 (1987), 349--358.

\bibitem{HardtSimon1989} Hardt, R.; Simon, L., {\sl Nodal sets for solutions of elliptic equations}, J. Differential Geom. 30 (1989), no. 2, 505--522.

\bibitem{Holopainen1992} Holopainen, I., {\sl Positive solutions of quasilinear elliptic equations on Riemannian manifolds}, Proc. London Math. Soc. (3) 65 (1992), no. 3, 651--672.


\bibitem{Holopainen1999} Holopainen, I.,{\sl Volume growth, Green's functions, and parabolicity of ends}, Duke Math. J. 97 (1999), no. 2, 319--346.

\bibitem{HuiskenIlmanen1999} Huisken, G.; Ilmanen, T.,{\sl  The inverse mean curvature flow and the Riemannian Penrose inequality}, J. Differential Geom. 59 (2001), no. 3, 353--437.

\bibitem{KeselmanZorich1996}
Kesel'man, V. M.; Zorich V. A., {\sl On the conformal type of a Riemannian manifold} (in
Russian), Funktsional. Anal. i Prilozhen. 30 (1996), 40--55, 96; English transl. in
Functional Anal. Appl. 30 (1996), 106--117.


\bibitem{KichenassamyVeron1986} Kichenassamy, S.; V\`{e}ron, L., {\sl Singular solutions of the p-Laplace equation}. Math. Ann. 275 (1986), no. 4, 599–615.




\bibitem{KotschwarNi2009}Kotschwar, B.; Ni, L., {\sl Local gradient estimates of $p$-harmonic functions, $1/H$-flow, and an entropy formula}, Ann. Sci. \'{E}c. Norm. Sup\'{e}r. (4) 42 (2009), no. 1, 1--36.

\bibitem{Lewis77} Lewis, J. L., {\sl Capacitary functions in convex rings}, Arch. Rational Mech. Anal. 66 (1977), no. 3, 201--224.

\bibitem{LiTam1992} Li, P.; Tam, L.-F., {\sl Green's functions, harmonic functions, and volume comparison}, J. Differential Geom. 41 (1995), no. 2, 277--318.

\bibitem{Lin1991} Lin, F.-H. {\sl Nodal sets of solutions of elliptic and parabolic equations}, Comm. Pure Appl. Math. 44 (1991), no. 3, 287--308.

\bibitem{ManfrediWeitsman1988} Manfredi, J.; Weitsman, A., {\sl On the Fatou theorem for $p$-harmonic functions}, Communications in Partial Differential Equations 13 (1988), pp. 651--668.

\bibitem{MariRigoliSetti2019} Mari, L.; Rigoli, M.;Setti, A. G.,{\sl On the $1/H$-flow by $p$-laplace approximation: new estimates via fake distances under Ricci lower bounds}, preprint, arXiv:1905.00216, to appear in Amer. J. Math.

\bibitem{Moser2007} Moser, R., {\sl The inverse mean curvature flow and p-harmonic functions}, J. Eur. Math. Soc. (JEMS) 9 (2007), no. 1, 77--83.

\bibitem{Moser2008} Moser, R., {\sl The inverse mean curvature flow as an obstacle problem}, Indiana Univ. Math. J. 57 (2008), no. 5, 2235--2256.

\bibitem{Moser2015} Moser, R., {\sl Geroch monotonicity and the construction of weak solutions of the inverse mean curvature flow}, Asian J. Math. 19 (2015), no. 2, 357--376.

\bibitem{MunteanuWangL2019} Munteanu, O.; Wang, L.,{\sl  Gradient estimate for harmonic functions on K\"ahler manifolds}, Trans. Amer. Math. Soc. 372 (2019), no. 12, 8759--8791.

\bibitem{MunteanuWang2021} Munteanu, O.; Wang, J.,{\sl Comparison theorems for three-dimensional manifolds with scalar curvature bounds}, preprint, arXiv:2105.12103.

\bibitem{MunteanuWang2022} Munteanu, O.; Wang, J.,{\sl Comparison theorems for 3D manifolds with scalar curvature bound, II}, preprint, arXiv:2201.05595.

\bibitem{NaberValtorta2014} Naber, A.; Valtorta, D., {\sl Sharp estimates on the first eigenvalue of the p-Laplacian with negative Ricci lower bound}, Math. Z. 277 (2014), no. 3--4, 867--891.

\bibitem{Stern2019} Stern, D., {\sl Scalar curvature and harmonic maps to $\mathbb{S}^1$}, preprint, arXiv:1908.09754, to appear in J. Differential Geom.

\bibitem{SungWang2014} Sung, C. A.; Wang, J., {\sl  Sharp gradient estimate and spectral rigidity for p-Laplacian}. Math. Res. Lett. 21 (2014), no. 4, 885–904.


\bibitem{Tolksdorf1984} Tolksdorf, P., {\sl Regularity for a more general class of quasilinear elliptic equations}. J. Differential Equations 51 (1984), no. 1, 126--150.

\bibitem{Varopoulos1981} Varopoulos, N., {\sl The Poisson kernel on positively curved manifolds}, J. Funct. Anal. 44 (1981), 359--380.

\bibitem{WangZhang2011} Wang, X.; Zhang, L., {\sl Local gradient estimate for $p$-harmonic functions on Riemannian manifolds}, Comm. Anal. Geom. 19 (2011), no. 4, 759--771.

\end{thebibliography}
\end{document}